\documentclass[a4paper,10pt]{amsart}

\usepackage[utf8]{inputenc}
\usepackage[T1]{fontenc}
\usepackage{lmodern}
\usepackage{amsmath}
\usepackage{amssymb}
\usepackage{amsthm}
\usepackage{mathrsfs}
\usepackage{tikz}
\usetikzlibrary{matrix}
\usetikzlibrary{positioning}
\usetikzlibrary{trees}
\usetikzlibrary{automata}

\newtheorem{theorem}{Theorem}[section]
\newtheorem{conjecture}[theorem]{Conjecture}
\newtheorem{question}{Question}
\newtheorem{proposition}[theorem]{Proposition}
\newtheorem{lemma}[theorem]{Lemma}
\newtheorem{hypotheseh}{Hypothesis}
\newtheorem*{hypotheseH}{Hypothesis H}
\newtheorem{fact}[theorem]{Fact}
\newtheorem{corollary}[theorem]{Corollary}

\theoremstyle{definition}
\newtheorem{definition}[theorem]{Definition}
\newtheorem{remark}[theorem]{Remark}

\renewcommand{\d}{\mathrm{d}}
\renewcommand{\phi}{\varphi}
\renewcommand{\epsilon}{\varepsilon}

\newcommand{\Id}{\mathrm{Id}}
\newcommand{\im}{\mathrm{im}}

\newcommand{\Vect}{\mathrm{Span}}
\newcommand{\dom}{\mathrm{supp}}
\newcommand{\bb}{\mathbb}
\renewcommand{\t}{\text}

\newcommand{\rank}{\mathrm{rk}}
\renewcommand{\index}{\operatorname{index}}

\title{On growth rate and contact homology}

\author{Anne Vaugon}
\address{Anne Vaugon, Laboratoire de Mathématiques Jean Leray, 2, rue de la Houssinière, BP 92208,
44322 Nantes Cedex 3, France }
\email{anne.vaugon@univ-nantes.fr}
\begin{document}
\begin{abstract}
It is a conjecture of Colin and Honda that the number of periodic Reeb orbits of universally tight contact structures on hyperbolic manifolds grows exponentially with the period, and they speculate further that the growth rate of contact homology is polynomial on non-hyperbolic geometries. Along the line of the conjecture, for manifolds with a hyperbolic component that fibers on the circle, we prove that there are infinitely many non-isomorphic contact structures for which the number of periodic Reeb orbits of any non-degenerate Reeb vector field grows exponentially. Our result hinges on the exponential growth of contact homology which we derive as well. We also compute contact homology in some non-hyperbolic cases that exhibit polynomial growth, namely those of universally tight contact structures on a circle bundle non-transverse to the fibers.
\end{abstract}
\maketitle

\section{Introduction and main results}
The goal of this paper is to study connections between the asymptotic number of periodic Reeb orbits of a $3$-dimensional contact manifold and the geometry of the underlying manifold. We first recall some basic definitions of contact geometry.
A $1$-form $\alpha$ on a $3$-manifold $M$ is called a \emph{contact form} if $\alpha\wedge\d \alpha$ is a volume form on $M$. A \emph{(cooriented) contact structure} $\xi$ is a plane field defined as the kernel of a contact form.
If $M$ is oriented, the contact structure $\ker(\alpha)$ is called \emph{positive} if the $3$-form $\alpha\wedge\d \alpha$ orients $M$. 
The \emph{Reeb vector field} associated to a contact form $\alpha$ is the vector field $R_\alpha$ such that $\iota_{R_\alpha}\alpha=1$ and $\iota_{R_\alpha}\d\alpha=0$.
It strongly depends on $\alpha$. The Reeb vector field (or the associated contact form) is called \emph{non-degenerate} if all periodic orbits are non-degenerate ($1$ is not an eigenvalue of the differential of the first return map).

A fundamental step in the classification of contact structures was the definition of tight and overtwisted contact structures given by Eliashberg \cite{Eliashberg89} following Bennequin \cite{Bennequin83}. A contact structure $\xi$ is said to be \emph{overtwisted} if there exists an embedded disk tangent to $\xi$ on its boundary. Otherwise $\xi$ is said to be \emph{tight}. \emph{Universally tight} contact structures are structures admitting a tight lift on universal cover. A contact form is called \emph{hypertight} if there is no contractible periodic Reeb orbit. Universally tight and hypertight~\cite{Hofer93} contact structures are always tight. 

To study a contact structure, it is useful to focus on the periodic orbits of a Reeb vector field. Weinstein conjectured that every contact form on a closed manifold admits a periodic orbit and that was proved in dimension $3$ by Taubes~\cite{Taubes07}.
Beyond the existence of a single periodic Reeb orbit, Colin and Honda are interested in the number  $N_L(\alpha)$ of periodic Reeb orbits with period at most $L$. They believe it is related to one of the Thurston geometries of the underlying manifold, namely the hyperbolic geometry.

\begin{conjecture}[Colin-Honda {\cite[Conjecture 2.10]{ColinHonda08}}]\label{conjecture_hyperbolique}
For all non-degenerate contact forms~$\alpha$ of a universally tight contact structure on a hyperbolic closed $3$-manifold, $N_L(\alpha)$ exhibits an exponential growth.
\end{conjecture}

The main result of this paper is related to Conjecture~\ref{conjecture_hyperbolique} and applies to manifolds with a non-trivial JSJ decomposition including a hyperbolic component that fibers on the circle (see \cite{BBBMP10} for more information).

\begin{theorem}\label{theoreme_hyperbolique}
Let $M$ be a closed oriented connected $3$-manifold which can be cut along a nonempty family of incompressible tori into irreducible manifolds including a hyperbolic component that fibers on the circle. Then, $M$ carries an infinite number of non-isomorphic, hypertight, universally tight contact structures such that for all hypertight non-degenerate contact forms $\alpha$, $N_L(\alpha)$ grows exponentially with~$L$.
Additionally, if the full contact homology is well defined and invariant and if $\alpha$ is only non-degenerate then $N_L(\alpha)$ grows exponentially with $L$.
\end{theorem}

Currently, contact homology is not defined in full generality. In what follows this assumption will be called Hypothesis H. The proof of the first part of Theorem~\ref{theoreme_hyperbolique} uses only a well-defined contact homology : our assumptions on $\alpha$ assure that the contact homology is always well-defined. However, the second part of the theorem depends on Hypothesis H. See Section~\ref{section_contact_homology} for more details. The fibration condition in Theorem~\ref{theoreme_hyperbolique} is not an insurmountable restriction as Agol~\cite{Agol12} recently proved the virtually fibered conjecture~\cite{Thurston82} which says that a hyperbolic $3$-manifold fibers on the  circle up to finite covering.
Note that the situation in Theorem~\ref{theoreme_hyperbolique} is different from the situation in Conjecture~\ref{conjecture_hyperbolique}.

\emph{Contact homology} and more generally \emph{Symplectic Field Theory (SFT)} are invariants of the contact structure introduced by Eliashberg, Givental and Hofer in 2000~\cite{EGH00}. This is a Floer homology invariant.
The associated complex is generated by periodic Reeb orbits and the differential ``counts'' pseudo-holomorphic curves in the symplectization.
Besides full contact homology, there exist two simpler contact homologies : \emph{cylindrical contact homology}~\cite{BEE09} and \emph{linearized contact homology} which depends on a given ``augmentation''.
Computation of contact homology hinges on finding periodic orbits and and pseudo-holomorphic curve by solving elliptic partial differential equations and this is usually out of reach. 
The \emph{growth rate of contact homology} is an invariant derived from the cylindrical or linearized contact homology introduced by Bourgeois and Colin~\cite{BourgeoisColin05}. It describes the asymptotic behavior with $L$ of the number of periodic Reeb orbits with period smaller than $L$ that contribute to contact homology. It is the contact equivalent of the growth rate of symplectic homology introduced by Seidel~\cite{Seidel08} and used by McLean~\cite{McLean10} to distinguish between cotangent bundles and smooth affine varieties.
Theorem~\ref{theoreme_hyperbolique} is a corollary of Theorem~\ref{theoreme_hyperbolique_homologie}.
\begin{theorem}\label{theoreme_hyperbolique_homologie}
Under the hypothesis of Theorem \ref{theoreme_hyperbolique}, the manifold $M$ carries an infinite number of non-isomorphic, hypertight, universally tight contact structures with an exponential growth rate of cylindrical contact homology restricted to primitive classes.
Under Hypothesis H, the growth rate of linearized contact homology is exponential for any pull-back of the trivial augmentation.
\end{theorem}
The ``pull-back'' augmentations will be explained in Section~\ref{section_lch}.
This theorem draws its inspiration in Colin and Honda's results~\cite{ColinHonda08} on exponential growth of contact homology for contact structures adapted to an open book with pseudo-Anosov monodromy. As proved by Thurston~\cite{Thurston98}, a manifold that fibers on the circle is hyperbolic if and only if it is the suspension of a surface by a diffeomorphism homotopic to a pseudo-Anosov map.

Colin and Honda speculate further that the growth rate of contact homology is polynomial in the non-hyperbolic situations. These situations are the following.
\begin{enumerate}
  \item On manifolds with spherical geometry, the growth rate of contact homology for universally tight contact structures is linear.
  \item On manifolds with a geometric structure neither hyperbolic nor spherical, the growth rate of contact homology for universally tight contact structures is usually polynomial.
\end{enumerate}
There is however already an exception to the second situation we will soon discuss.

In this paper, we study contact structures in a non-hyperbolic situation. We make use of Giroux~\cite{Giroux01} and Honda~\cite{Honda1_00} classification of positive contact structures on circle bundles to try out Colin and Honda's conjectures on a broad family of contact structures. Let $\pi : M\to S$ be a circle bundle over a closed compact surface. Figure~\ref{figure_classification_fibres_cercles} gives a summary of this classification. Statements such as ``tangent to the fibers'' or ``transverse to the fibers'' mean that there exists an isotopic contact structure with this property. Additionally, $\chi(S)$ is the Euler characteristic and $\chi(S,M)$ the Euler number of the fibration.
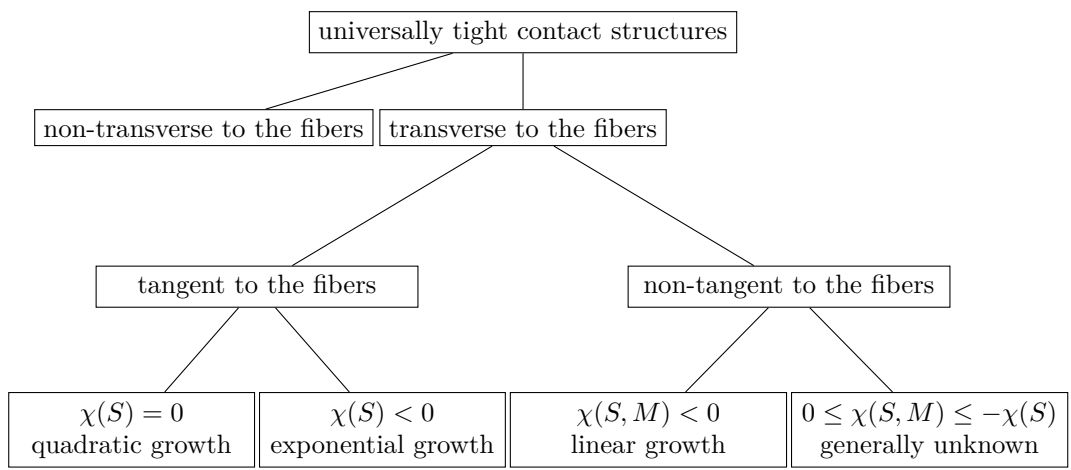
\begin{figure}[here]
\begin{tikzpicture}[
	every node/.style = {draw, text badly centered}
]
 \node(A)
 	{ universally tight contact structures
    }
 [sibling distance=4.2cm, level distance=3\baselineskip]
 child { node (B)
	{ non-transverse to the fibers
	}
 }
 child { node (C)
 	{ transverse to the fibers
	}
    [sibling distance=7cm, level distance=5\baselineskip, text width=4cm]
   child { node (D)
   		{tangent to the fibers
   		}
   	   [sibling distance=3.3cm, level distance=4.5\baselineskip, text width=3cm]
       child { node (E)
		{ $\chi(S)=0$\\
		quadratic growth
		}
       }
       child { node (F)
   		{ $\chi(S)<0$\\
   		exponential growth
   		}
       }
   }
    child { node (G)
   		{ non-tangent to the fibers
   		}
   	   [sibling distance=3.7cm, level distance=4.5\baselineskip, text width=3.4cm]
       child { node (H)
		{ $\chi(S,M)< 0$\\
		linear growth
		}
       }
       child { node (I)
   		{ $0\leq\chi(S,M)\leq-\chi(S)$\\
   		generally unknown
   		}
       }
   }
 }
    child[missing] { node {} };
\end{tikzpicture}
\caption{Universally tight positive contact structures on circle bundles over a surface $S$ with non-positive Euler characteristic.}\label{figure_classification_fibres_cercles}
\end{figure}

In some cases, the contact homology and its growth rate are already known. For instance, contact structures tangent to the fibers are fiberwise covering of $(UTS, \xi_\text{std})$ where $UTS$ is the unit tangent bundle over $S$ and $\xi_\text{std}$ is the contact element contact structure (see for example \cite{Geiges08}). In this case, the Reeb flow of the standard contact form associated to a Riemannian metric is the geodesic flow. If the surface is hyperbolic, there exists an unique closed geodesic in each homotopy class \cite[Theorem 3.9.5]{Klingenberg82} and the number of homotopy classes has exponential growth with respect to length~\cite{Milnor68}. Therefore, growth rates of the number of periodic Reeb orbits and of contact homology are exponential. This is an exception to the second statement of the conjecture of Colin and Honda in the non-hyperbolic cases.

If $S$ is a torus, universally tight contact structures are standard contact structures on $T^3$ \cite{Giroux00}. The contact homology is known~\cite{EGH00} and its growth rate is quadratic, see for example Bourgeois's Morse-Bott approach to contact homology~\cite{Bourgeois02}. Bourgeois also studied contact structures transverse and non-tangent to the fibers with $\chi(S,M)< 0$. He computed contact homology and obtained a linear growth rate. 
Each of these contact structure has an $S^1$-invariant contact structure in its isotopy class.

In this paper we study the other cases where contact structures are universally tight and non-transverse to the fibers. 
\begin{definition}[Giroux \cite{Giroux01}]
A contact structure $\xi$ on a fiber bundle $\pi : M\to S$ is \emph{walled} by an oriented multi-curve $\Gamma$ on $S$ if
  \begin{enumerate}
    \item $\xi$ is transverse to the fibers on  $M\setminus \pi^{-1}(\Gamma)$,
    \item $\xi$ is transverse to $\pi^{-1}(\Gamma)$ and tangent to fibers of $\pi^{-1}(\Gamma)$. We call $\pi^{-1}(\Gamma)$ a \emph{wall}.
  \end{enumerate}
\end{definition}

Note that walled contact structures admit an $S^1$-invariant walled contact structure in their isotopy class. The following theorem justifies the previous definition.

\begin{theorem}[Giroux {\cite[Théorème 4.4]{Giroux01}}]\label{theorem_wall}
Universally tight positive contact structures non-transverse to the fibers are exactly contact structures isotopic to a contact structure walled by a non-trivial multi-curve that contains no contractible component.
\end{theorem}

Giroux's definition of walled contact structures and Theorem \ref{theorem_wall} provide us a way to decompose our manifold into understandable pieces. This brings us to the second half of this paper.

\begin{theorem}\label{theoreme_cloisonne}
Let $\pi : M\to S$ be a circle bundle over a closed compact surface and $\xi$ be a positive contact structure walled by a non-trivial multi-curve $\Gamma=\bigcup_{i=0}^n\Gamma_i$ that contains no contractible component. Let $X=M\setminus \pi^{-1}(\Gamma)$ be the complement of the wall. Denote by $X^+_1,\dots, X^+_{n_+}$ the connected components of $X$ for which $\xi$ is positively transverse to the fibers and by $X^-_1,\dots, X^-_{n_-}$ those for which $\xi$ is negatively transverse to the fibers.
Let $\eta$ be a loop in $M$. Then, there exists a hypertight contact form $\alpha$ such that the cylindrical contact homology $HC_*^{[\eta]}(M,\alpha,\mathbb Q)$ is well defined and we have the following
\begin{enumerate}
 \item\label{C1} if $[\eta]=[\text{fiber}]^k$ for $\pm k>0$, then $HC_*^{[\eta]}(M,\alpha,\mathbb Q)=\displaystyle{\bigoplus_{j=1}^{n_\pm}} H_*(X^\pm_j,\mathbb Q)$,
 \item\label{C2} if there exist $j\in \{1,\dots,n\}$ and $k'>0$ such that $[\pi(\eta)]=[\Gamma_j]^{k'}$, then $HC_*^{[\eta]}(M,\alpha,\mathbb Q)=\displaystyle{\bigoplus_{I}} H_*(S^1,\mathbb Q)$ where $I=\{i; [\pi(\Gamma_i)]=[\pi(\Gamma_j)]  \}$,
 \item\label{C3} otherwise, $HC_*^{[\eta]}(M,\alpha,\mathbb Q)=0$.
\end{enumerate}
In addition, if Hypothesis H is satisfied, the growth rate of the contact homology $HC(M,\xi,\mathbb Q)$ is quadratic.
\end{theorem}
Here $HC_*^{[\eta]}(M,\alpha,\mathbb Q)$ denotes the cylindrical contact homology restricted to the homotopy class $[\eta]$.
To obtain the whole picture for universally tight positive contact structures on fiber bundles, it remains to compute contact homology of contact structures transverse to the fibers with $0\leq\chi(S,M)\leq-\chi(S)$. Although this has not been dealt with yet, it seems reasonable to work out.
However, Colin and Honda's questions remain out of reach as we have yet a lot to learn about contact structures on hyperbolic manifolds. 
On the other hand, for non-hyperbolic geometries there is already a counterexample as observed above. The following question may provide some alternative way to tackle connections between geometry and periodic Reeb orbits.
\begin{question}
Is the growth rate of contact homology for a given contact structure related to that of the fundamental group of the underlying manifold ? For example, can the growth rate of the fundamental group be an upper bound for the growth rate of contact homology ? (Roughly speaking, in a finitely generated group, the growth rate counts the number of elements that can be written as a product of length $n$.)
\end{question}
\begin{question}
Are there growth rates of contact homology that lie between quadratic and exponential growths ?
\end{question}

This paper is part of the authors's the PhD thesis~\cite{Vaugon11}. It is organized as follows. In Section~\ref{section_contact_homology}, we briefly outline the theory of contact homology. Morse-Bott approach to contact homology is sketched in Section~\ref{section_Morse_Bott}. In Section~\ref{section_croissance}, we give a detailed definition of the growth rate of contact homology. In Section~\ref{section_positivity_intersection} we discuss positivity of intersection for tori foliated by Reeb orbits. We prove Theorem~\ref{theoreme_cloisonne} in Section~\ref{section_polynomial} and Theorem~\ref{theoreme_hyperbolique} in Section~\ref{section_hyperbolic}.

\bigskip
\noindent{\bf{Acknowledgements.}} I am deeply grateful to my advisor, Vincent Colin, for his guidance and support. I would also like to thank Frédéric Bourgeois, Paolo Ghiggini, Patrick Massot, François Laudenbach and Chris Wendl for stimulating and helpful discussions. Thanks also go Jean-Claude Sikorav, Jacqui Espina and an anonymous referee for suggesting numerous improvements and corrections and Marc Mezzarobba for proofreading this text. Last, I am grateful for the hospitality of the Unité de Mathématiques Pures et Appliquées (\textsc{ens} Lyon).

\section{Contact homology}\label{section_contact_homology}
Throughout this paper, we consider only manifolds of dimension $3$ and cooriented contact structures. In this section, we give an overview of contact homology over $\mathbb Q$ which was introduced by Eliashberg, Givental and Hofer~\cite{EGH00}. This is an homology build around Gromov's holomorphic curves~\cite{Gromov85} which he introduced to the symplectic world in 1985. For more details see \cite{EGH00, Bourgeois02}.
The reader is reminded that the existence and invariance of contact homology in full generality is still in progress. We will elaborate a little bit more within this section, one can also refer to~\cite{FFGW12}. We start with reviewing some basics.

\subsection{Almost-complex structures and holomorphic curves}
The \emph{symplectization} of a contact manifold $(M,\xi=\ker(\alpha))$ is the non-compact symplectic manifold $(\mathbb R\times M,\d(e^\tau \alpha)=\omega)$, where $\tau$ is the $\mathbb R$-coordinate.
An \emph{almost complex structure} on a even-dimensional manifold $X$ is a map $J :TX\to TX$ preserving the fibers and such that $J^2=-\Id$. In addition, on a symplectization, $J$ is \emph{adapted} to $\alpha$ if 
$J$ is $\tau$-invariant,
$J\frac{\partial}{\partial \tau}=R_\alpha$,
$J\xi=\xi$ and
$J$ is \emph{compatible} with $\omega$ (i.e. $\omega(\cdot, J\cdot)$ is a Riemannian metric).

We are interested in \emph{$J$-holomorphic curves}. These are curves $u:(\Sigma,j)\to \bb (\mathbb R\times M,J)$ such that $\d u\circ j=J\circ \d u$ where $(\Sigma,j)$ is a Riemannian surface. This equation is called the \emph{Cauchy-Riemann equation}. When $J$ is unspecified we call $u$ a \emph{pseudo-holomorphic} curve. One can refer to \cite{McDuffSalamon04} for more information.

\begin{theorem}[see for instance {\cite[Lemma 2.4.1]{McDuffSalamon04}}]\label{theoreme_zeros_isoles}
Let $U$ be an open subset of a Riemann surface $(S,j)$ and let $(M,J)$ be a manifold with an almost complex structure. Then, the critical points of any non-constant $J$-holomorphic map $u: (U,j)\to (M,J)$ are isolated.
\end{theorem}

To define contact homology, we consider pseudo-holomorphic maps $u :(\dot{\Sigma},j)\to \mathbb R\times M$ where $(\dot{\Sigma},j)$ is a punctured Riemannian surface. For example, the simplest non-constant holomorphic maps are trivial cylinders: if $\gamma$ is a $T$-periodic Reeb orbit, the associated \emph{trivial cylinder} is the map
\begin{equation*}
\begin{array}{ccc}
 \bb R\times S^1&\longrightarrow & \mathbb R\times M\\
(s,t)&\longmapsto& (Ts,\gamma(Tt)).
\end{array}
\end{equation*}
Let $x$ be a puncture of $\dot{\Sigma}$ and, for some neighborhood of $x$, choose some polar coordinates $(\rho,\theta)$ centered at $x$.
Such a map $u=(u_{\mathbb R},u_M)$ is called \emph{positively asymptotic} to a $T$-periodic
orbit $\gamma$ in a neighborhood of $x$ if $\lim_{\rho\to 0} u_{\mathbb R}(\rho,\theta)=+\infty$ and $\lim_{\rho\to 0} u_M(\rho,\theta)=\gamma\left(- T\theta\right)$.
Similarly, it is called \emph{negatively asymptotic} to $\gamma$ if $\lim_{\rho\to 0} u_{\mathbb R}(\rho,\theta)=-\infty$ and $\lim_{\rho\to 0} u_M(\rho,\theta)=\gamma\left(+ T\theta\right)$.

Holomorphic curves $u:(\dot{\Sigma},j)\to \bb (R\times M,J)$ with finite Hofer energy are asymptotic to periodic Reeb orbits near each puncture~\cite[Theorem 1.3]{HWZ96}. Recall that the \emph{Hofer energy} $E$ of $u:\dot{\Sigma}\to \bb R\times M$ is
\[E_\alpha(u)=\sup\left\{\int_{\dot{\Sigma}} u^*\d(\phi\alpha), \phi : \bb R\to [0,1], \phi'\geq 0\right\}.\]

The following proposition is used in Section \ref{section_positivity_intersection} to prove the smoothness of the projection of a holomorphic curve on $M$.
\begin{proposition}[{\cite[Theorem 4.1]{HWZ95}}]\label{proposition_tau_im_du}
Let $(M,\xi=\ker(\alpha))$ be a contact manifold and $J$ an almost complex structure on $(\mathbb R\times M,\d (e^\tau\alpha))$ adapted to $\alpha$. Consider the standard complex structure $j$ on $\mathbb R\times S^1$. For every non-constant map $u:(\mathbb R\times S^1,j)\to\mathbb (R\times M,J)$ that is not a trivial cylinder, the points $(s,t)$ such that $\frac{\partial}{\partial \tau}\in\im(\d u(s,t)) $ 
are isolated. 
\end{proposition}

\subsection{Full contact homology}

\subsubsection{Periodic orbits}
Let $\gamma$ be a $T$-periodic Reeb orbit of a contact manifold $(M,\xi=\ker(\alpha))$ and let $p\in\gamma$. Denote the Reeb flow by $\phi_t$. The linearized return map preserves the contact structure. Its restriction $\psi_T$ to $\xi_p$ is a symplectomorpism of $(\xi_p,\d\alpha)$. Recall that a non-degenerate periodic orbit $\gamma$ is called \emph{even} if $\psi_T(p)$ has two real positive eigenvalues and \emph{odd} if $\psi_T(p)$ has two complex conjugate or two real negative eigenvalues. Let $\gamma^m$ be the $m$-th multiple of a simple orbit $\gamma^1$. Then $\gamma^m$ is said to be \emph{good} if $\gamma^1$ and $\gamma^m$ have the same parity, otherwise, $\gamma^m$ is said to be \emph{bad}. 

A relative grading of non-degenerate periodic Reeb orbit is given by the \emph{Conley-Zehnder} index. This is a Maslov type index. Additionally, its parity matches with the definitions of odd and even periodic orbits. We refer to \cite{Laudenbach04} for a precise presentation.

\subsubsection{Definition of full contact homology}\label{section_def_HC}
Let $(M,\xi=\ker(\alpha))$ be a contact manifold with a non-degenerate contact form. 
The chain complex $A_*(M,\alpha)$ is the free super-commutative unital $\bb Q$-algebra generated by all good periodic Reeb orbits i.e. the simple periodic orbits and their good multiples. 
Choose an almost complex structure $J$ adapted to $\alpha$. To define the differential $\partial$, consider the moduli space $\mathcal M_{[Z]}(J,\gamma,\gamma'_1,\dots\gamma'_n)$. This is the space of equivalent classes of solutions to Cauchy-Riemann equation that have finite energy, are positively asymptotic to $\gamma$, negatively asymptotic to $\gamma'_1\dots\gamma'_n$ and in the relative homotopy class $[Z]$. By equivalence classes we mean modulo reparametrization of $\dot{\Sigma}$. The $\mathbb R$-translation in $\mathbb R\times M$ induces a $\mathbb R$-action on $\mathcal M_{[Z]}(J,\gamma,\gamma'_1,\dots\gamma'_n) $ (see \cite{Bourgeois03} for more details). The differential counts elements in $\mathcal M_{[Z]}(J,\gamma,\gamma'_1,\dots\gamma'_n)/\bb R$ when this space is $0$-dimensional.
To define a homology, we must have $\partial\circ\partial=0$. We want to apply Morse homology type arguments. However, $\mathcal M_{[Z]}(J,\gamma,\gamma'~,\dots\gamma'_n)/\bb R$ is not a manifold and we have to assume Hypothesis~\ref{H1}. We denote by $\mu_{[Z]}(\gamma,\gamma'_1,\dots\gamma'_n)$ the Conley-Zehnder index of $\gamma$ minus the sum of the Conley-Zehnder indices of $\gamma'_1,\dots\gamma'_n$ where all indices are calculated with respect to a trivialization on $Z$.

\begin{hypotheseh}\label{H1}
There exists an abstract perturbation of the Cauchy-Riemann equation such that $\mathcal M_{[Z]}(J,\gamma,\gamma'~,\dots\gamma'_n)/\bb R$
is a union of branched labeled manifolds with corners and rational weights, having the expected dimension $\mu_{[Z]}(\gamma,\gamma'_1,\dots\gamma'_n)+n-2$.
\end{hypotheseh}

The problems arise when we take into account multiply covered curves. There are many approaches to proving that the relevant moduli spaces indeed have such a structure as in Hypothesis~\ref{H1} due to Fukaya and Ono~\cite{FukayaOno99}, Liu and Tian~\cite{LiuTian98}, Hofer, Wysocki and Zehnder~\cite{Hofer08,HWZ,HWZ11}. There also exist partial transversality results due to Dragnev~\cite{Dragnev04}.

Let $n_{\gamma,\gamma'_1,\dots\gamma'_n}$ denote the signed weighted counts of points in $0$-dimensional components of $\mathcal M_{[Z]}(J,\gamma,\gamma'~,\dots\gamma'_n)/\bb R$ for all relative homology classes $[Z]$~\cite{EGH00,BourgeoisMohnke03}. The signs correspond to the orientation of the moduli space and the weights to the multiplicity of the orbits.
The differential of a periodic orbit $\gamma$ is
\begin{equation*}\label{definition_partial}
\partial \gamma=\sum_{\gamma'_1,\dots,\gamma'_n} \frac{n_{\gamma,\gamma'_1,\dots,\gamma'_n}}{i_1!\dots i_l ! \kappa(\gamma'_1)\dots\kappa(\gamma'_n)}\gamma'_1\dots\gamma'_n.
\end{equation*} 
The sum is taken over all the sets of periodic Reeb orbits and we divide by the overcount: $i_1\dots i_l$ are multiplicities in $\{\gamma'_1\dots\gamma'_n\}$ and $\kappa(\gamma)$ is the multiplicity of $\gamma$. The definition is extended using the graded Leibniz rule.

Hypothesis~\ref{H1} guarantees $\partial\circ\partial=0$. Next, we still want a contact invariant and so we need to change the contact form. We construct a symplectic cobordism between these two forms.
Let $\alpha_1$ and $\alpha_0$ be two non-degenerate, homotopic contact forms. Then there exist $c>0$ and a family $(\alpha_l)_{l\in \bb R}$ such that
\begin{enumerate}
  \item $\displaystyle\lim_{l\to-\infty} \alpha_l=c\alpha_0$,
  \item $\displaystyle\lim_{l\to\infty} \alpha_l=\alpha_1$,
  \item let $\alpha$ denote the $1$-form on $\bb R\times M$ induced by  $(\alpha_l)_{l\in \bb R}$, then $\d\alpha\wedge\d\alpha> 0$.
\end{enumerate}
Choose almost complex structures $J_1$ and $J_0$ on $\bb R\times M$ adapted to $\alpha_1$ and $\alpha_0$. We denote by $J^c$ the almost complex structure such that $J^c_{\vert \xi}=J_{\vert \xi}$ and $J^c\frac{\partial}{\partial \tau}=\frac{R_\alpha}{c}$.
Last, choose an almost complex structure $J$  on $\bb R\times M$ compatible with $\d\alpha$ and interpolating between $J_1$ and $J_0^c$.
Let $\gamma$ be a $R_{\alpha_1}$-periodic orbit and $\gamma'_1,\dots,\gamma'_n$ be $R_{c\alpha_0}$-periodic orbits. Let  $\mathcal M_{[Z]}(J,\gamma,\gamma'_1,\dots\gamma'_n)$ denote the moduli space of $J$-holomorphic curves, positively asymptotic to $\gamma$, negatively asymptotic to $\gamma'_1\dots\gamma'_n$ and in the relative homotopy class $[Z]$. We have to assume Hypothesis~\ref{H2} to obtain a nice structure on these moduli spaces. 

\begin{hypotheseh}\label{H2}
There exists an abstract perturbation of the Cauchy-Riemann equation such that
$\mathcal M_{[Z]}(J,\gamma,\gamma'_1,\dots\gamma'_n)$
is a union of branched labeled manifolds with corners and rational weights, having the expected dimension $\mu_{[Z]}(\gamma,\gamma'_1,\dots\gamma'_n)+n-1$.
\end{hypotheseh}

Assuming Hypothesis~\ref{H2}, there exists a chain map 
\begin{equation*} 
\psi((\alpha_1,J_1),(c\alpha_0,J_0^c)) :(A_*(M,\alpha_1),\partial_{J_1})\to(A_*(M,c\alpha_0),\partial_{J_0^c})
\end{equation*}
similar to $\partial$ and counting elements in $\mathcal M_{[Z]}(J,\gamma,\gamma'_1,\dots\gamma'_n)$ when this space is $0$-dimensional.

The induced map in homology
\begin{equation*}
\Psi((\alpha_1,J_1),(c\alpha_0,J_0^c)) :HC_*(M,\alpha_1,J_1)\to HC_*(M,c\alpha_0,J_0^c)
\end{equation*}
does not depend on $\alpha_l$ or $J$. The map $\Psi$ is the key ingredient to obtain invariance of contact homology.

\begin{hypotheseH}
Hypothesis H is the union of Hypotheses~\ref{H1} and~\ref{H2}.
\end{hypotheseH}

\begin{theorem}[Eliashberg-Givental-Hofer]\label{theoreme_partial_2} Let $(M,\xi)$ be a closed contact manifold and $\alpha$ be a non-degenerate contact form. Under Hypothesis H,
\begin{enumerate}
 \item $\partial^2=0$,
 \item the associated homology $HC_*(M,\xi)$ does not depend on the choice of the contact form, complex structure and abstract perturbation. 
 \end{enumerate}
\end{theorem}
One can consult \cite{Bourgeois09} for a sketch of proof.
If $\partial^2=0$ for some contact form $\alpha$, we denote $HC_*(M,\alpha,J)$ the associated homology.
Some computations were carried out by Bourgeois and Colin \cite{BourgeoisColin05} to distinguish toroidal irreducible $3$-manifolds, Ustilovsky \cite{Ustilovsky99} to prove the existence of exotic contact structures on spheres and Yau \cite{Yau06} who proved that the contact homology of overtwisted contact structures is trivial. Bourgeois \cite{Bourgeois02} provided other computations using Morse-Bott approach to contact homology.

\begin{remark}
Under Hypothesis~\ref{H1}, it is reasonable for us to expect the following. Suppose all the images of $J$-holomorphic buildings positively asymptotic to $\gamma$, negatively asymptotic to $\gamma'_1\dots\gamma'_n$ are contained in an open set $U$ in $ \mathbb R\times M$. Then $U$ contains the images of all solutions of perturbed Cauchy-Riemann equations that have the same asymptotics for all abstract perturbations. Roughly speaking, holomorphic buildings are glued holomorphic curves and these are defined in~\cite{BEE09} in more details. We will only apply this result when the holomorphic buildings are holomorphic cylinders. 
\end{remark}

\subsubsection{Composition of cobordism maps}\label{section_contact_form}
Consider the special case of proportional contact forms $\alpha$ and $c\alpha$ for $c>0$. Let $J$ be an almost complex structure on $\mathbb R\times M$ adapted to $\alpha$. Then we have the diffeomorphism
\begin{equation*}
\begin{array}{cccc}
\phi_c:& \bb R\times M&\longrightarrow &\bb R\times M \\
&(\tau,x)&\longmapsto& (c\tau,x)
\end{array}
\end{equation*}
which sends a $J$-holomorphic curve to a $J^c$-holomorphic curve. The identification of geometric periodic Reeb orbits induces an isomorphism
\begin{equation*} 
\theta(\alpha,J,c) :(A_*(M,\alpha),\partial_{J})\to (A_*(M,c\alpha),\partial_{J^c}).
\end{equation*}
Let
\begin{equation*} 
\Theta(\alpha,J,c) :HC(M,\alpha,\partial_{J})\to HC_*(M,c\alpha,\partial_{J^c})
\end{equation*}
 denote the induced map on homology.
The maps $\Psi$ and $\Theta$ have natural composition properties.

\begin{theorem}[Eliashberg-Givental-Hofer]\label{theoreme_Phi_0_1}
Let $(\alpha_i, J_i)$ be contact forms and adapted almost complex structures on the symplectization of a closed manifold for $i=0,1,2$.
Under Hypothesis~H, if there exist cobordisms as defined in Section~\ref{section_def_HC} between $(\alpha_2,J_2)$ and $(\alpha_1,J_1)$ and between $(\alpha_1,J_1)$ and $(\alpha_0,J_0)$, then \[\Psi((\alpha_2,J_2),(\alpha_0,J_0))=\Psi((\alpha_1,J_1),(\alpha_0,J_0))\circ\Psi((\alpha_2,J_2),(\alpha_1,J_1)).\]
\end{theorem}

\begin{proposition}\label{theoreme_Phi_0_1_1}
Let $(\alpha_i, J_i)$ be contact forms and adapted almost complex structures on the symplectization of a closed manifold for $i=0,1$.
Under Hypothesis~H,
\begin{enumerate}
  \item\label{prop_c_positif} for all $c>0$, we have the following \[\Theta(\alpha_0,J_0,c)\circ\Psi((\alpha_1,J_1),(\alpha_0,J_0))=\Psi((c\alpha_1,J^c_1),(c\alpha_0,J^c_0))\circ\Theta(\alpha_1,J_1,c),\]
  \item\label{prop_c_et_theta} if $c<1$, one can choose $\psi((\alpha_0,J_0),(c\alpha,J_0^c))=\theta(\alpha,J_0,c)$.
\end{enumerate}
\end{proposition}

\begin{proof}[Sketch of proof]
(\ref{prop_c_positif}) Denote by $\alpha_l$ and $J$ the homotopy and almost complex structure used to define $\psi((\alpha_1,J_1),(\alpha_0,J_0))$. Consider the homotopy $c\alpha_\frac{l}{c}$ and the almost complex structure $J^c=\phi_* J$ where $\phi :(\tau,x)\mapsto (c\tau,x)$. Then $\phi$ sends $J$-holomorphic curves to $J^c$-holomorphic curves.

(\ref{prop_c_et_theta}) Consider the homotopy $\alpha_l=c(l)\alpha_0$ between $c\alpha_0$ and $\alpha_0$ where $c$ is a non-decreasing function. Let $J_0$ be an almost complex structure adapted to $\alpha_0$ and $C$ be an anti-derivative of $c$. The almost complex structure $J=\phi_*J_0$ where  $\phi:(\tau,x)\longmapsto (C(\tau),x)$
is adapted to $\alpha_l$. Then $\phi$ sends $J_0$-holomorphic curves to $J$-holomorphic-curves and the $J_0$-holomorphic curves used to define \[\psi((\alpha_0, J_0),(\alpha_0, J_0))\] are trivial cylinders.
\end{proof}

\subsection{Cylindrical contact homology}
Let $(M,\xi=\ker(\alpha))$ be a closed contact manifold and assume $\alpha$ is a non-degenerate and hypertight. The chain complex $C^\t{cyl}_*(M,\alpha)$ of \emph{cylindrical contact homology} is the $\bb Q$-vector space generated by good periodic Reeb orbits associated to the form $\alpha$. Choose an almost complex structure $J$ on $\mathbb R\times M$ adapted to $\alpha$. The cylindrical differential $\partial^\t{cyl}$ counts $J$-holomorphic rigid cylinders in the moduli space $\mathcal M_{[Z]}(J,\gamma,\gamma')/\mathbb R$ defined in Section~\ref{section_def_HC}.
The differential of a periodic orbit $\gamma$ is
\begin{equation*}
\partial^\t{cyl} \gamma=\sum_{\gamma'} \frac{n_{\gamma,\gamma'}}{\kappa(\gamma')}\gamma'
\end{equation*}
where the sum runs over $\gamma'$ such that $\gamma$ and $\gamma'$ are of index difference $1$ and $n_{\gamma,\gamma'}$ and $\kappa$ are defined in Section~\ref{section_def_HC}.
As we do in the full contact homology version, we assume Hypothesis H for a nice structure on the moduli spaces.

\begin{theorem}[Eliashberg-Givental-Hofer]
Let $(M,\xi)$ be a closed contact manifold and $\alpha$ be a non-degenerate hypertight contact form. Under Hypothesis H,
\begin{enumerate}
 \item $(\partial^\t{cyl})^2=0$;
 \item the associated homology $HC^{cyl}_*(M,\xi)$ does not depend on the choice of an hypertight contact form $\alpha$, an almost complex structure $J$ and an abstract perturbation.
 \end{enumerate}
\end{theorem}
One can consult \cite{Bourgeois09} for a sketch of proof.
We may however restrict cylindrical contact homology to a subset of homotopy classes so that Hypothesis H is satisfied. The differential of cylindrical contact homology respects homotopy classes so we can always restrict it. If $\Lambda$ is a set of free homotopy classes of $M$ we call \emph{partial cylindrical homology restricted to $\Lambda$} the homology of the chain complex $(C^\Lambda_*(M,\alpha),\partial)$ where $C^\Lambda_*(M,\alpha)$ is generated by good periodic Reeb orbits in $\Lambda$.
If $\Lambda$ contains only primitive free homotopy classes, it is shown in~\cite[Corollary~1.9]{Dragnev04} that for a generic almost complex structure, the partial contact homology $HC_*^\Lambda(M,\alpha,J)$ is well defined and does not depend on the choice of $J$ or of a hypertight non-degenerate form. 

\begin{fact}
The morphisms from Theorem \ref{theoreme_Phi_0_1} induce morphisms $\psi^\t{cyl}$ and $\Psi^\t{cyl}$ on the cylindrical contact homology complex and on cylindrical contact homology with similar properties.
\end{fact}

\subsection{Linearized contact homology}\label{section_lch}

Cylindrical contact homology is a special case of linearized contact homology. Introduced in Chekanov's work on Legendrian contact homology \cite{Chekanov02}, linearized contact homology was generalized to contact homology by Bourgeois, Ekholm and Eliashberg~\cite{BEE09}. Linearization in SFT also appears in Cieliebak and Latschev's work~\cite{CL09}. One can also refer to \cite{Bourgeois09,ColinHonda08}.

\begin{definition}
An \emph{augmentation} $\epsilon : (A,\partial)\to (\bb Q,0)$ is a $\bb Q$-algebra homomorphism that is also a chain map.
\end{definition}

An augmentation $\epsilon$ in $(A,\partial)$ gives a ``change of coordinates'' $a\mapsto \overline{a}=a-\epsilon(a)$. Let $(A^\epsilon(M,\alpha),\partial^\epsilon)$ denote the new chain complex and write $\partial^\epsilon=\partial^\epsilon_1+\partial^\epsilon_2+\dots $ using the filtration by word length. In particular $\partial^\epsilon_0=0$.
\begin{proposition}
If~$\epsilon$ is an augmentation, then $(\partial^\epsilon_1)^2=0$.
\end{proposition}

Let $(M,\xi=\ker(\alpha))$ be a contact manifold with a non-degenerate contact form and $\epsilon$ be an augmentation of $A_*(M,\alpha)$. 
The \emph{linearized contact homology} with respect to $\epsilon$, $HC^\epsilon(M,\alpha,J)$, is the homology of $(A^\epsilon_*(M,\alpha),\partial^\epsilon_1)$ where $A^\epsilon_*(M,\alpha) $ is the $\bb Q$-vector space generated by $\{\overline{\gamma}, \gamma \text{ good period orbit}\}$.

\begin{proposition}\label{proposition_trivial_augmentation}
Under Hypothesis H, if the contact form $\alpha$ is hypertight, the complex $A_*(M,\alpha)$ admits the trivial augmentation. The linearized contact homology is then the cylindrical contact homology.
\end{proposition}

Let $\alpha_0$ and $\alpha_1$ be two non-degenerate, homotopic contact forms and \[\phi:(A(M,\alpha_1),\partial_{J_1})\to(A(M,\alpha_0),\partial_{J_0})\] be a chain map. If $\epsilon_0$ is an augmentation on $(A(M,\alpha_0),\partial_{J_0})$ then $\phi$ induces a \emph{pull back augmentation} $\epsilon_1=\epsilon_0\circ \phi$ on $(A(M,\alpha_1),\partial_{J_1})$.  The morphisms $\psi$ and $\Psi$ described in Theorem \ref{theoreme_Phi_0_1} induce morphisms $\psi^{\epsilon_0}$ and $\Psi^{\epsilon_0}$. We define $\theta^{\epsilon_0}$ and $\Theta^{\epsilon_0}$ in the same way.

\begin{theorem}[see {\cite[Theorem 2.8]{Bourgeois09}}]\label{theorem_augmentation}
Under Hypothesis H,
\begin{enumerate}
 \item the set of linearized contact homologies
\begin{equation*}
\{HC^\epsilon(M,\alpha,J),\epsilon \t{ augmentation of } (A_*(M,\alpha),\partial_J) \}
\end{equation*}
is an invariant of the isotopy class of the contact structure $\xi=\ker(\alpha)$,
\item let $\phi_1, \phi_2 :(A(M,\alpha_1),\partial_{J_1})\to(A(M,\alpha_0),\partial_{J_0})$ be two homotopic chain maps and $\epsilon_0$ be an augmentation on $(A(M,\alpha_0),\partial_{J_0})$. Let $\epsilon_1$ and $\epsilon_2$ denote the pull-back augmentations by $\phi_1$ and $\phi_2$.
Then, the map
\begin{equation*}
\begin{array}{cccc}
\phi(\epsilon_1,\epsilon_2):&(A^{\epsilon_1}_*(M,\alpha_1),\partial_{J_1})&\longrightarrow&(A^{\epsilon_2}_*(M,\alpha_1),\partial_{J_1})\\
&\gamma-\epsilon_1(\gamma)&\longmapsto& \gamma-\epsilon_2(\gamma)
\end{array}
\end{equation*}
induces an isomorphism $\Phi(\epsilon_1,\epsilon_2)$ in homology such that the diagram
\begin{center}
 \begin{tikzpicture}
  \draw(0,0)   node    (A)  {$HC^{\epsilon_1}_*(M,\alpha_1,J_1)$};
  \draw(4,0)  node    (B)   {$HC^{\epsilon_2}_*(M,\alpha_1,J_1) $};
  \draw(2,-1.3)    node     (C)  {$HC^{\epsilon_0}_*(M,\alpha_0,J_0)$};

  \path[->] (A) edge              node [left]       {$\Phi_1$} (C)
                edge              node  [above] {$\Phi(\epsilon_1,\epsilon_2) $ } (B)
            (B) edge              node  [right]      {$\Phi_2$ } (C);
 \end{tikzpicture}
\end{center}
 commutes where $\Phi_1$ and $\Phi_2$ are the morphisms induced by~$\phi_1$ and~$\phi_2$.
\end{enumerate}
\end{theorem}
Augmentations $\epsilon_1$ and $\epsilon_2$ are said to be \emph{homotopic}, see \cite[Section 2.5]{Bourgeois09} for a general definition.

\section{Morse-Bott approach to contact homology}\label{section_Morse_Bott}

Bourgeois introduced Morse-Bott approach to contact homology in his PhD thesis~\cite{Bourgeois02} in 2002. It gives a way to compute contact homology when the contact form is degenerate and there exist submanifolds foliated by periodic Reeb orbits. The main idea is to compare the Morse-Bott degenerate situation to non-degenerate situations obtained by perturbing the degenerate form using a Morse function. In this text, we will only use part of the theory on simple examples to compute the contact homology of circle bundles.

\subsection{Perturbation of contact forms of Morse-Bott type}
Let $(M,\xi=\ker(\alpha))$ be a contact manifold with a contact form $\alpha$ and let $\phi_t$ be the Reeb flow.
\begin{definition}
The form $\alpha$ is of \emph{Morse-Bott type} if
  \begin{enumerate}
    \item the set $\sigma(\alpha)$ of period of periodic Reeb orbits is discrete, $\sigma(\alpha)$ is called the \emph{action spectrum},
    \item if $L\in\sigma(\alpha)$, then $N_L=\{p\in M, \phi_L(p)=p\}$ is a smooth closed submanifold;
    \item the rank of $\d\alpha_{\vert N_L}$ is locally constant and $T_pN_L=\ker(\d\phi_L-I)$.
  \end{enumerate}
\end{definition}

For instance, the standard contact form $\alpha_n=\sin(nx)\d y+\cos(nx)\d z$ on $T^3$ is of Morse-Bott type. The Reeb vector field is 
\[R_{\alpha_n}=\left(\begin{array}{c}
0\\
\sin(nx)\\
\cos(nx)
\end{array}\right)\]
and its flow preserves all tori $\{x=\t{cst}\}$. A torus $\{x=x_0\}$ is foliated by periodic Reeb orbits if and only if $ \sin(nx_0)$ and $\cos(nx_0)$ are rationally dependent. Another important example is the case of a contact structure transverse to the fibers on a circle bundle and $S^1$-invariant: such a contact structure admits a contact form whose Reeb vector field is tangent to the fibers. The whole manifold is then foliated by periodic Reeb orbits of the same period.

The Reeb flow induces an $S^1$-action on $N_L$ for all $L\in\sigma(\alpha)$. In general, the quotient space $S_L$ is an orbifold. However in the examples studied in this paper, the spaces $S_L$ will be smooth manifolds. Hence, we assume here that $S_L$ is smooth.

We now describe how to perturb a contact form $\alpha$ of Morse-Bott type. Fix $L\in\sigma(\alpha)$. For all $L'\in\sigma(\alpha)\cap[0,L]$, choose a Morse function $f_{L'}$ on $S_{L'}$ and extend it to $N_{L'}$ so that $\d f_{L'}(R_\alpha)=0$. Then, extend it to $M$ using cut-off functions in such a way that its support is contained in a small neighborhood of $N_{L'}$. Let $\overline{f}_L$ denote the sum of all these functions. Perturb the contact form to
$\alpha_{\lambda,L}=(1+\lambda\overline{f}_L)\alpha$.

\begin{proposition}[Bourgeois {\cite[Lemma 2.3]{Bourgeois02}}]
For all $L>0$, there exists $\lambda(L)>0$ such that for all $0<\lambda\leq\lambda(L)$, the periodic orbits of $\alpha_{\lambda,L}$ with period smaller than $L$ correspond to critical points of $f_{L'}$ on $S_{L'}$ for $L'\in\sigma(\alpha)\cap[0,L]$. Additionally, these periodic orbits are non-degenerate.
\end{proposition}

\begin{remark}
An almost complex structure $J$ on the symplectization $\bb R\times M$ which is $S^1$-invariant on $N_L$ induces a Riemannian metric $\d\alpha(\cdot,J\cdot)$ on $S_L$.
\end{remark}

\subsection{Morse-Bott contact homology}

Roughly speaking, the complex of Morse-Bott approach to contact homology is generated by critical points of the functions $f_L$, and the differential counts generalized holomorphic cylinders. Generalized holomorphic cylinders are a combination of holomorphic curves asymptotic to periodic orbits in the spaces $N_L$ and gradient lines in the spaces $S_L$. See \cite{Bourgeois02} for more details, \cite{Bourgeois102} for a summary of \cite{Bourgeois02}, or \cite{Bourgeois03} for a general presentation.

Consider a family of almost complex structures $J_\lambda$ adapted to 
$\alpha_{\lambda,L}$ and $S^1$-invariant on $N_{L'}$ for all $L'\leq L$. Generalized holomorphic cylinders are limits of $J_\lambda$-holomorphic curves as $\lambda\to 0$ and derive from two main phenomena. On one side, holomorphic buildings appear similarly to the non-degenerate situation: up to reparametrization, a sequence converges in $\mathcal C^\infty$-loc to a holomorphic curve with asymptotic periodic orbits in some intermediate spaces $N_L$. On the other hand, when the asymptotics of two adjacent levels in a holomorphic building differ, projections on $S_L$ grow nearer to a gradient trajectory of $f_L$: up to reparametrization, a sequence converges in $\mathcal C^\infty$-loc to a trivial cylinder over any point of the gradient trajectory. The associated compactness theorem derives from Bourgeois's thesis \cite[Chapters 3 and 4]{Bourgeois02}. One can also refer to~\cite{BEHWZ03}. In our simpler setting, Bourgeois's results lead to the following theorems.

\begin{theorem}[Bourgeois \cite{Bourgeois02}]\label{theoreme_MB_S_1_invariant_modules} Let $\pi:M\to S$ be a circle bundle over a closed oriented surface carrying an $S^1$-invariant contact form $\alpha$ transverse to the fibers. Fix $L>0$ and a Morse-Bott perturbation $f_L$ induced by a Morse function $f:S\to\bb R$. Let $J_\lambda$ be a family of $S^1$-invariant almost complex structures on $\bb R\times M$ adapted to $\alpha_{\lambda,L}$ and converging to an almost complex structure $J$ adapted to $\alpha$ as $\lambda\to 0$. Assume that $(f,g)$ is a Morse-Smale pair where $g$ is the Riemannian metric on $S$ induced by $J$ and $\alpha$.
Fix two critical points $x_+ $et $x_-$ of $f$ so that $\index(x_+)-\index(x_-)=1$ and let $\gamma_+$ and $\gamma_-$ denote the $k$-th iterates of associated simple periodic Reeb orbits. Then, for all small $\lambda$, the moduli space
$\mathcal M(\gamma_+,\gamma_-,J_\lambda)/\bb R$ is a $0$-dimensional manifold. Additionally, $\mathcal M(\gamma_+,\gamma_-,J_\lambda)/\bb R$ can be identified with the set of gradient trajectories from $x_+$ to $x_-$, the holomorphic curves are arbitrarily close to cylinders over the gradient trajectories and the orientations induced by contact homology and Morse theory are the same.
\end{theorem}

\begin{theorem}[Bourgeois \cite{Bourgeois02}]\label{theoreme_MB_T_3_invariant_modules}
Consider the standard contact form \[\alpha=\sin( x)\d y+\cos( x)\d z\] on $T^3$. Fix $L>0$ and a Morse-Bott perturbation $f_L$ induced by a Morse function $f:S^1\to\bb R$ with two critical points. Let $J_\lambda$  be a family of almost complex structures on $\bb R\times M$ adapted to $\alpha_{\lambda,L}$, $S^1$-invariant on $N_{L'}$ for all $L'\leq L$ and converging to an almost complex structure $J$.
Fix $L'\leq L$ and let $T$ be a torus in $N_{L'}$. Let $\gamma_+$ and $\gamma_-$ denote the $k$-th iterates of two simple periodic orbits in $T$ associated to the critical points of $f$. Then for all small enough  $\lambda$, the moduli space
$\mathcal M(\gamma_+,\gamma_-,J_\lambda)/\bb R$
has exactly two elements with opposite orientations and the holomorphic curves are arbitrarily close to cylinders over gradient trajectories of $f$. In addition, if $\gamma_+$ and $\gamma_-$ are not in the same Morse-Bott torus, $\mathcal M(\gamma_+,\gamma_-,J_\lambda)/\bb R $ is empty.
\end{theorem}

\begin{remark}
This theorem generalizes to contact forms $\sin(n x)\d y+\cos(n x)\d z$ and $f(x)\d y+g(x)\d z$ if $f$ and $g$ are increasing and decreasing on the same sets as $ x\mapsto \sin(n x)$ and $x\mapsto \cos(n x)$.
\end{remark}
These theorems derive from Bourgeois's work and Hypothesis~\ref{H1} is always satisfied. Indeed, the solutions of the Cauchy-Riemann equations is the $0$-set of a Fredholm section in a Banach bundle (described in \cite[5.1.1]{Bourgeois02}) and thus a $3$-manifold. To achieve transversality of this section, Bourgeois proves that the linearized Cauchy-Riemann operator is surjective on its $0$-set by studying its surjectivity for curves close to holomorphic curves (the curves are defined in \cite[5.3.2]{Bourgeois02}, the surjectivity is proved in \cite[Proposition 4.13 and 5.14]{Bourgeois02}) and then using an implicit function theorem \cite[Proposition 5.16]{Bourgeois02}. To obtain the desired moduli space, we quotient the space of solutions by the biholomorphisms of $\bb R\times S^1$ and the $\bb R$-action. The orientation issues are studied in  \cite[Proposition 7.6]{Bourgeois02}.

\begin{corollary}[Bourgeois \cite{Bourgeois02}]\label{corolary_bundle}
Let $M$ be an oriented circle bundle over a closed oriented surface $S$ carrying an $S^1$-invariant contact structure $\xi$ which is transverse to the fibers. Let $f$ denote the homotopy class of the fiber. Then, for all $k>0$, there exists a contact form $\alpha$ such that
\[HC_*^{[f^k]}(M,\alpha,\mathbb Q)=H_*(S,\mathbb Q).\]
The cylindrical contact homology is trivial in all other homotopy classes.
\end{corollary}

\begin{corollary}\label{proposition_homologie_T_3}
Fix $n\in\mathbb N^*$.
Let $\alpha_n=\sin(nx)\d y+\cos(nx)\d z $ be the standard contact form on $T^3$. Let $c_y$ and $c_z$ denote the free homotopy classes associated to $\{0\}\times\{0\}\times S^1$ and $\{0\}\times S^1\times \{0\}$. Let $\eta=c_y^{n_y} c_z^{n_{\vphantom{y}z}}$ be a non-trivial homotopy class.
Then there exists a contact form $\alpha'_n$ such that
\[HC_*^{[\eta]}(T^3,\alpha'_n,\mathbb Q)=\bigoplus_{i=1}^n H_*(S^1,\bb Q).\]
The cylindrical contact homology is trivial in all other homotopy classes.
\end{corollary}

Note that contact homology distinguishes between the contact structures $\ker(\alpha_n)$.
Following Corollaries~\ref{corolary_bundle} and \ref{proposition_homologie_T_3} and assuming Hypothesis~\ref{H2}, we obtain that the growth rate of contact homology is linear in the circle bundle case and quadratic for $T^3$. In this paper, we will consider the case where the contact manifold is obtained by gluing together pieces from these two examples.
We now turn to the definition of the growth rate of contact homology.
\section{Growth rate of contact homology}\label{section_croissance}

\subsection{Algebraic setting}

The \emph{growth rate} of a function $f :\mathbb R_+\to \mathbb R_+$ is said to be \emph{polynomial of order  $\leq n$} if there exists $a>0$ such that $f(x)\leq ax^n$ for all ~$x\in\mathbb R_+$. It is said to be \emph{exponential} if there exist $a>0$ and $b>0$ such that $f(x)\geq a\exp(bx)$ for all~$x\in\mathbb R_+$. More generally, we can define the \emph{growth rate type} of a function: two non-deceasing functions $f :\mathbb R_+\to \mathbb R_+$ and $g :\mathbb R_+\to \mathbb R_+$ have the same growth rate type if there exists $C>0$ such that
\[h\left(\frac{x}{C}\right)\leq g(x)\leq h(Cx)\]
for all~$x\in\mathbb R_+$ (see for instance \cite{delaHarpe00}). We call two such function \emph{equivalent} and denote by  $\Gamma(f)$ the associated equivalence class. With this definition, if $f$ is equivalent to a polynomial of degree $n$ its growth rate is polynomial of order $n$ and if $f$ is equivalent to an exponential, its growth rate is exponential. Note that our definition gives us more precise informations that the growth rate of symplectic homology~\cite{McLean10} given by the formula
\[\limsup_{x\to\infty} \frac{\log\big(\max(f(x),1)\big)}{\log(x)}\] and commonly used in other parts of topology~\cite{KatokHasselblatt95}.

The following algebraic preliminaries are similar to \cite{McLean10}. 
A \emph{filtered directed system} is a family of vector spaces $(E_x)_{x\in[0,\infty)}$ such that for all $x_1\leq x_2$, there exists a linear map $\phi_{x_1,x_2}: E_{x_1}\longrightarrow E_{x_2}$ such that 
\begin{enumerate}
  \item  $\phi_{x_1,x_1}=\Id$  for all $x_1\geq 0$,
  \item $\phi_{x_1,x_3}=\phi_{x_2,x_3}\circ\phi_{x_1,x_2}$ for all $0\leq x_1\leq x_2 \leq x_3$.
\end{enumerate}
A filtered directed system admits a direct limit $E=\lim_{x\to \infty} E_x$. By definition, there exist maps $\phi_x:E_x\to E$ such that the following diagram commutes for all $0\leq x_1\leq x_2$.
\begin{center}
 \begin{tikzpicture}[shorten >=1pt,node distance=2cm,auto]

  \node      (A)                       {$E_{x_1}$};
  \node      (B) [below right of=A]  {$E$};
  \node      (C) [above right  of=B] {$E_{x_2}$};

  \path[->] (A) edge              node        {$\phi_{x_1,x_2}$} (C)
                  edge            node [swap] {$\phi_{x_1} $ } (B)
            (C) edge              node        {$\phi_{x_2} $ } (B);
 \end{tikzpicture}
\end{center}
In what follows, we will assume that $E_x$ is a finite dimensional space for all $x\geq 0$.
\begin{definition}
The \emph{growth rate} $\Gamma((E_x)) $ of $(E_x)$ is the growth rate of $x\mapsto\rank(\phi_x)$.
\end{definition} 

A \emph{morphism} of filtered directed systems from $(E_x)_{x\in[0,\infty)}$ to $(F_x)_{x\in[0,\infty)}$ consists of a positive number $C$ and a family of linear maps $\Phi_x :E_x\longrightarrow F_{Cx}$ such that the following diagram commutes for all $0\leq x_1\leq x_2$
\begin{center}
 \begin{tikzpicture}[shorten >=1pt,auto]
  \node      (A)                       {$E_{x_1}$};
  \node      (B) [right=of A]         {$F_{Cx_1}$};
  \node      (C) [below=of A]         {$E_{x_2}$};
  \node      (D) [right=of C]         {$F_{Cx_2}$};

  \path[->] (A) edge              node [swap] {$\phi^E_{x_1,x_2}$} (C)
                edge              node        {$\Phi_{x_1} $ } (B)
             (C)   edge              node        {$\Phi_{x_2} $ } (D)
            (B) edge              node        {$\phi^F_{Cx_1,Cx_2}$} (D);
 \end{tikzpicture}
\end{center}
Two systems $(E_x)$ and $(F_x)$ are \emph{isomorphic} if there exists a morphism $(C,\Phi)$ from $(E_x)$ to $(F_x)$ and a morphism $(C',\Psi)$ from $(F_x)$ to $(E_x)$ such that, for all $x\geq 0$, \[\Psi_{Cx}\circ\Phi_x=\phi^E_{x,CC'x} \text{ and } \Phi_{C'x}\circ\Psi_x=\phi^F_{x,CC'x}.\]

\begin{lemma}\label{lemme_croissance}
Two isomorphic filtered directed systems have the same growth rate.
\end{lemma}

\begin{proof}
Consider two filtered directed systems $(E_x)$ and $(F_x)$.
By definition, the following diagram
\begin{center}
 \begin{tikzpicture}[shorten >=1pt,auto]
\useasboundingbox (-3,-3.7) rectangle (2.5,0.8);
  \node   (A)                      {$E_{x_1}$};
  \node   (B) [right=of A]         {$\lim E$};
  \node   (C) [below=of A]         {$F_{Cx_1}$};
  \node   (D) [below=of B]         {$\lim F$};
  \node   (E) [below=of C]         {$E_{CC'x_1}$};
  \node   (F) [below=of D]         {$\lim E$}; 
  \node  (F)  [below=of D]         {$\lim E$};
   \node  (G) at(-3,-1.5)        {$\Id$};
   
  \path[->] (A) edge              node [swap] {} (C)
                edge              node  [swap]      {$\phi_{x_1}$} (B)
                edge[bend left,out=-70,in=-120] node[swap]{$\phi_{x,CC'x}$}(E)
            (C) edge              node [swap] { } (E)
                edge              node  [swap]      {$\psi_{Cx_1}$} (D)
            (B) edge              node        {$u$} (D)
            (D) edge              node        {$v$} (F)
            (E) edge              node        {} (F);
            
 \draw   [->] (B) .. controls (-4,3) and (-4,-6) .. (F);
       
 \end{tikzpicture}
\end{center}
 commutes. Thus $\rank(\phi_{x_1})\leq\rank(\psi_{Cx_1})$. Similarly $\rank(\psi_{x_1})\leq\rank(\phi_{C'x_1})$.
\end{proof}

\subsection{Action filtration}

Let $M$ be a compact manifold and $\alpha$ be a hypertight non-degenerate contact form on $M$. Fix $L>0$ and let $C_{\leq L}^\t{cyl}(M,\alpha)$ be the $\bb Q$-vector space generated by the good periodic Reeb orbits with period smaller than $L$. This is a finite dimensional vector space. Under Hypothesis H, since the differential decreases the action, $(C_{\leq L}^\t{cyl}(M,\alpha),\partial_{\leq L})_{L>0}$ is a chain complex. We denote by $(HC_{\leq L}^\t{cyl}(M,\alpha,J))_{L>0}$ the associated homology. The inclusion \[i: C_{\leq L}^\t{cyl}(M,\alpha)\longrightarrow C_{\leq L'}^\t{cyl}(M,\alpha)\] induces a linear map in homology for all $L'\geq L$. Similarly, given a set of free homotopy classes $\Lambda$, for all $L>0$ we define a chain complex $(C_{\leq L}^{\Lambda}(M,\alpha),\partial_{\leq L})$ and a homology $HC_{\leq L}^\Lambda(M,\alpha,J)$.
\begin{fact}
The families $(HC_{\leq L}^\t{cyl}(M,\alpha,J))_{L>0}$ and $(HC_{\leq L}^\Lambda(M,\alpha,J))_{L>0}$ are filtered directed systems whose morphisms are induced by inclusions.
Furthermore
\begin{align*}
\lim_{\to} HC_{\leq L}^\t{cyl}(M,\alpha,J)&=HC_*^\t{cyl}(M,\alpha,J)\\
\lim_{\to}HC_{\leq L}^\Lambda(M,\alpha,J)&=HC_*^\Lambda(M,\alpha,J).
\end{align*}
\end{fact}
Let $M$ be a compact manifold and $\alpha$ be a non-degenerate contact form on $M$ such that
$(C_*(M,\alpha),\partial)$ admits an augmentation $\epsilon$. Then $\partial^\epsilon _1$ decreases the action on $A^\epsilon(M,\alpha)$ and we can define a filtered directed system.
\begin{definition}
The \emph{growth rate} of contact homology is the growth rate of the associated filtered directed system.
\end{definition}
\begin{remark}\label{remark_growth}
As
$\rank(\phi_{L})\leq\dim HC_{\leq L}^{\t{cyl}}(M,\alpha,J)\leq\dim C_{\leq L}^{\t{cyl}}(M,\alpha)$,
if the growth rate of contact homology is exponential, the number of periodic Reeb orbits grows exponentially with the period.
\end{remark}

\subsection{Invariance of the growth rate of contact homology}

\begin{fact}\label{fait_filtration}
The maps from Section~\ref{section_def_HC} restrict to maps denoted 
\begin{align*}
\psi_{\leq L}((\alpha_1,J_1),(c\alpha_0,J_0^c))\\
\Psi_{\leq L}((\alpha_1,J_1),(c\alpha_0,J_0^c))
\end{align*}
in the filtered case.
In addition $\theta(\alpha, J,c)$ and $\Theta(\alpha, J,c)$ restrict to maps
\begin{align*}
\theta_{\leq L}(\alpha, J,c) :&(A_{\leq L}(M,\alpha),\partial_J)\to (A_{\leq cL}(M,c\alpha),\partial_{J^c})\\
\Theta_{\leq L}(\alpha, J,c) :&HC_{\leq L}(M,\alpha,J)\to HC_{\leq cL}(M,c\alpha,J^c).
\end{align*}
Analogous restrictions exist in the cylindrical and linearized situations.
\end{fact}

\begin{fact}\label{fait_filtration_1_2}
Let $\phi_1, \phi_2 :(A(M,\alpha),\partial_{J})\to(A(M,\alpha_0),\partial_{J_0})$ be two homotopic chain maps and $\epsilon_0$ be an augmentation on $(A(M,\alpha_0),\partial_{J_0})$. Let $\epsilon_1$ and $\epsilon_2$ denote the pull-back augmentations by $\phi_1$ and $\phi_2$. The map $\phi(\epsilon_1,\epsilon_2)$ from Theorem~\ref{theorem_augmentation} induces a map
\begin{equation*}
\Phi_{\leq L}(\epsilon_1,\epsilon_2) :HC_{\leq L}^{\epsilon_1}(M,\alpha,J)\to HC^{\epsilon_2}_{\leq L}(M,\alpha,J).
\end{equation*}
\end{fact}

\begin{fact}\label{fait_filtration_1}
Let $\epsilon_0$ be an augmentation on $(A(M,\alpha),\partial_{J})$
In addition to the properties from Theorem \ref{theoreme_Phi_0_1} and Proposition~\ref{theoreme_Phi_0_1_1}, the maps defined in Fact~\ref{fait_filtration} satisfy the following properties. 
\begin{enumerate}
  \item For all $0<c<1$,
  \begin{equation*}
  \Theta_{\leq L}\left(c\alpha, J^c,\frac{1}{c}\right)\circ\Psi_{\leq L}((\alpha,J),(c\alpha,  J ^c))
  \end{equation*}
  is the map induced by the inclusion $HC_{\leq L}(\alpha,J)\to HC_{\leq \frac{L}{c}}(\alpha,J)$.
  \item  If
  $\phi_1=\psi((\alpha_1,J_1),(c\alpha,J^c))\circ \psi((\alpha,J),(\alpha_1,J_1))$ and
  $\phi_2=\theta(\alpha,J,c)$
  then
  \begin{equation*}
  \Phi_{\leq \frac{L}{c}}(\epsilon_2,\epsilon_1)\circ\Theta^{\epsilon_2}_{\leq L}\left(c\alpha, J^c,\frac{1}{c}\right)\circ\Psi^{\epsilon_0}_{\leq L}((\alpha,J),(c\alpha,  J ^c))
  \end{equation*}
  is the morphism induced by the inclusion $HC^{\epsilon_1}_{\leq L}(\alpha,J)\to HC^{\epsilon_1}_{\leq \frac{L}{c}}(\alpha,J)$ where $\epsilon_1$ and $\epsilon_2$ denote the pull-back augmentations by $\phi_1$ and $\phi_2$.
\end{enumerate}
\end{fact}

\begin{proposition}\label{proposition_fds}
 Let $\alpha_0$ and $\alpha_1$ be two hypertight contact forms on a compact manifold $M$. Assume that $\alpha_0$ and $\alpha_1$ are homotopic through a family of contact forms. Under Hypothesis H, the two filtered directed systems $(HC_{\leq L}^\t{cyl}(M,\alpha_0))_{L\geq 0}$ and $(HC_{\leq L}^\t{cyl}(M,\alpha_1))_{L\geq 0}$ are isomorphic.
\end{proposition}

\begin{proof}
The morphisms between  $HC_{\leq L}^\t{cyl}(M,\alpha_1)$ and $HC_{\leq L}^\t{cyl}(M,\alpha_0)$ are
\begin{gather*}
\phi_L:HC_{\leq L}^\t{cyl}(M,\alpha_1)\to HC_{\leq \frac{L}{c}}^\t{cyl}(M,\alpha_0)\\
\phi_L= \Theta^\t{cyl}_{\leq L}\left(c\alpha_0,J_0^c,\frac{1}{c}\right)\circ\Psi^\t{cyl}_{\leq L}\left((\alpha_1,J_1), (c\alpha_0,J_0^c)\right)
\end{gather*} 
and
\begin{gather*}
\phi'_L:HC_{\leq L}^\t{cyl}(M,\alpha_0)\to HC_{\leq \frac{L}{c'}}^\t{cyl}(M,\alpha_1)\\
\phi'_L= \Psi^\t{cyl}_{\leq \frac{L}{c'}}\left(\left(\frac{\alpha_0}{c'},J_0^\frac{1}{c'}\right), (\alpha_1,J_1)\right)\circ\Theta^\t{cyl}_{\leq L}\left(\alpha_0,J_0,\frac{1}{c'}\right).
\end{gather*}
These morphisms give an isomorphism by Fact \ref{fait_filtration_1}.
\end{proof}

\begin{corollary}\label{proposition_invariance_croissance_cylindrique}
 Let $\alpha_0$ and $\alpha_1$ be two hypertight contact forms on a compact manifold $M$. Assume that $\alpha_0$ and $\alpha_1$ are homotopic through a family of contact forms. Under Hypothesis H, the associated cylindrical contact homologies have the same growth rate.
\end{corollary}

\begin{proposition}\label{proposition_invariance_croissance_partielle}
Let $\alpha_0$ and $\alpha_1$ be two hypertight contact forms on a compact manifold $M$. Assume that $\alpha_0$ and $\alpha_1$ are homotopic through a family of contact forms. Let $\Lambda$ be a set of primitive free homotopy classes of $M$. Then the associated cylindrical partial contact homologies have the same growth rate.
\end{proposition}

\begin{proof}
The restrictions to the primitive classes of the morphisms defined in the proof of Proposition \ref{proposition_fds} give an isomorphism between filtered directed systems. Apply Lemma \ref{lemme_croissance} to obtain the desired result.
\end{proof}

\begin{proposition}\label{proposition_invariance_croissance_linearisee}
 Let $\alpha_0$ and $\alpha_1$ be two isotopic contact forms, $J_0$ and $J_1$ be two adapted almost complex structures such that $(A_*(\alpha_0),\partial_{J_0})$ has an augmentation $\epsilon_0$ and $\psi((\alpha_1,J_1),(\alpha_0,J_0))$ exists. Let $\epsilon_1$ be the pull-back augmentation of $\epsilon_0$ (see Section \ref{section_lch}). Then, under Hypothesis H, the two filtered directed systems $(HC_{\leq L}^{\epsilon_1}(\alpha_1,J_1))_{L\geq 0}$ and $(HC_{\leq L}^{\epsilon_0}(\alpha_0,J_0))_{L\geq 0}$ are isomorphic. Thus, the growth rates of linearized contact homology are the same.
\end{proposition}

\begin{proof}Consider the morphisms
\begin{gather*}
\phi_L:HC_{\leq L}^{\epsilon_1}(M,\alpha_1,J_1)\to HC_{\leq L}^{\epsilon_0}(M,\alpha_0,J_0)\\
\phi_L= \Psi^{\epsilon_0}_{\leq L}\left((\alpha_1,J_1), (\alpha_0,J_0)\right)
\end{gather*} 
and
\begin{gather*}
\phi'_L:HC_{\leq L}^{\epsilon_0}(M,\alpha_0,J_0)\to HC_{\leq \frac{L}{c}}^{\epsilon_1}(M,\alpha_1,J_1)\\
\phi'_L= \Psi^{\epsilon_1}_{\leq \frac{L}{c}}\left(\left(\frac{\alpha_0}{c},J_0^\frac{1}{c}\right), (\alpha_1,J_1)\right)\circ\Phi_{\leq\frac{L}{c}}(\epsilon_0^c,\epsilon'_0)\circ\Theta^{\epsilon_0^c}_{\leq L}\left(\alpha_0,J_0,\frac{1}{c}\right)
\end{gather*}
where $\epsilon^c_0$ is the pull back augmentation of $\epsilon_0$ by $\theta\left(\frac{1}{c}\alpha_0,J^\frac{1}{c}_0,c\right)$ and $\epsilon'_0$ is the pull back by $\psi((\alpha_1,J_1),(\alpha_0,J_0))\circ \psi\left(\left(\frac{1}{c}\alpha_0,J^\frac{1}{c}_0,c\right),(\alpha_1,J_1)\right)$.
These morphisms give an isomorphism by Fact \ref{fait_filtration_1}.
\end{proof}

\section{Positivity of intersection and tori foliated by Reeb orbits}\label{section_positivity_intersection}

Introduced by Gromov \cite{Gromov85} and McDuff \cite{McDuff94}, positivity of intersection states that, in dimension $4$, two distinct pseudo-holomorphic curves $C$ and $C'$ have a finite number of intersection points and that each of these points contributes positively to the algebraic intersection number $C\cdot C'$. In this text, we will only consider the simplest form of positivity of intersection: let $M$ be a $4$-dimensional manifold, $C$ and $C'$ be two $J$-holomorphic curves and $p\in M$ so that $C$ and $C'$ intersect transversely at $p$. Consider $v\in T_p C$ and $v'\in T_p C'$ two non-zero tangent vectors. Then $(v,Jv,v',Jv')$ is an oriented basis of $T_p M$ ($J$ orients $T_p M$). In the contact world, positivity of intersection results in the following lemma.

\begin{lemma}
Let $(M,\xi)$ be a contact manifold, $\alpha$ be a contact form and $J$ be an adapted almost complex structure on the symplectization. Consider $U$ an open subset of $\mathbb C$, $u: U\to \mathbb R\times M$ a $J$-holomorphic curve and $z\in U$ such that $\d u_M(z)$ is injective and transverse to $R_\alpha(u_M(z))$. Then, $R_\alpha(u_M(z))$ is positively transverse to $\d u_M(z)$.
\end{lemma}

\begin{proof}
Let $\gamma : [-\epsilon,\epsilon]\to M$ be an arc in a Reeb trajectory such that $\gamma(0)=u_M(z)$. Consider the holomorphic curve
\[\begin{array}{cclc}
&v:\mathbb R\times [-\epsilon,\epsilon] &\longrightarrow & \mathbb R\times M\\
&(s,t)&\longmapsto& (s+u_{\mathbb R}(z),\gamma(t)).
\end{array}\]
The holomorphic curves $u$ and $v$ intersect transversely at $u(z)$ and $\left(\frac{\partial}{\partial \tau},R_\alpha(u_M(z))\right)$ is an oriented basis for the tangent plane to $v$ at $u(z)$. The projection of $u$ to $M$ is smooth as $\d u_M(z)$ is injective. Positivity of intersection gives the desired result.
\end{proof}
The hypothesis ``$\d u_M(z)$ injective and transverse to $R_\alpha(u_M(z))$'' is generic (see Theorem~\ref{theoreme_zeros_isoles} and Proposition \ref{proposition_tau_im_du}). 
We will use positivity of intersection in the following situation. Let $(M,\xi=\ker(\alpha))$ be a contact manifold with a chart $I\times S^1\times S^1$ where $I$ is an interval and coordinates $(x,y,z)$ such that $\alpha=f(x)\d y+g(x)\d z$. Assume that, for all $x\in I$, the tori $\{x\}\times S^1\times S^1=T_x$ are incompressible in $M$. A torus $T$ is \emph{incompressible} in $M$ if the map $\pi_1(T)\to\pi_1(M)$ is injective. Consider $u:\mathbb R\times S^1\to\mathbb R\times M$ a pseudo-holomorphic cylinder with finite energy and asymptotic to the periodic Reeb orbits $\gamma_+$ and $\gamma_-$ in $M$. Assume that $u$ intersects $\mathbb R\times T_x$ for all $x\in I$.

\begin{lemma}\label{lemme_u_-1}
There exists a nonempty open interval $I_1\subset I$ such that for all $x_0\in I_1$  \[u^{-1}\left(u(\mathbb R\times S^1)\cap(\mathbb R\times T_{x_0})\right)\] is a disjoint union of smooth circles homotopic to $\{*\}\times S^1$. 
\end{lemma}

\begin{proof}
There exists a nonempty open interval $I_1\subset I$ such that $I_1 \times S^1\times S^1$ does not intersect $\gamma_+$ and $\gamma_-$. Thus
$u^{-1}\left(u(\mathbb R\times S^1)\cap I \times S^1\times S^1\right)$ is contained in a compact subset of $\mathbb R\times S^1$.
As the points such that $\d u(s,t)=0$ or $\frac{\partial}{\partial \tau}\in \im(\d u(s,t))$ are isolated in $\mathbb R\times S^1$, we may assume that $I_1 \times S^1\times S^1$ does not contain images of points such that $\d u(s,t)=0$ or $\frac{\partial}{\partial \tau}\in \im(\d u(s,t))$.

Consider $x_0\in I_1$ and $(s,t)\in\mathbb R\times S^1$ such that $u(s,t)\in\mathbb R\times T_{x_0}$. As $\frac{\partial}{\partial \tau}\notin \im(\d u(s,t))$ and $u$ is pseudo-holomorphic, $R_\alpha(s,t)\notin \im(\d u(s,t))$ and 
\[\Vect\left(\frac{\partial}{\partial \tau},R_\alpha(u(s,t))\right)\cap\im(\d u(s,t))=\{0\}.\] 
As $\d u(s,t)\neq 0$, it holds that
\[\Vect\left(\frac{\partial}{\partial \tau},R_\alpha(u(s,t))\right)\oplus\im(\d u(s,t))=T_{u(s,t)}(\mathbb R\times M).\] 
Thus $\im(\d u)+T(\mathbb R\times T_{x_0})=T_{u}(\mathbb R\times M)$ and, by transversality, \[u^{-1}\left(u(\mathbb R\times S^1)\cap(\mathbb R\times T_{x_0})\right)\] is a $1$-dimensional compact submanifold of $\mathbb R\times S^1$.

By contradiction, if $u^{-1}\left(u(\mathbb R\times S^1)\cap(\mathbb R\times T_{x_0})\right)$ has a contractible component $C$, then $u(C)=c$ is contractible in $\mathbb R\times M$. As $c\subset\mathbb R\times T_{x_0}$ and $T_{x_0}$ is an incompressible torus, $c$ is contractible in $ \mathbb R\times T_{x_0}$. As $\Vect\left(\frac{\partial}{\partial \tau},R_\alpha(u(s,t))\right)\cap\im(\d u(s,t))=\{0\}$, the projection of $c$ to $M$ is smooth and transverse to $R_\alpha$. Yet the torus $T_{x_0}$ is foliated by Reeb orbits. Thus $u^{-1}\left(u(\mathbb R\times S^1)\cap(\mathbb R\times T_{x_0})\right)$ has only non-contractible components and, as it is a smooth manifold, these components are homotopic to $\{*\}\times S^1$.
\end{proof}
Let $C$ be a circle given by Lemma \ref{lemme_u_-1}, then $C$ inherits the orientation of $\{*\}\times S^1$ and induces a homotopy class of $T_{x_0}$. Let $p$ be a vector tangent to $T_{x_0}$ so that the straight line in $T_{x_0}$ directed by $p$ is in the homotopy class $[C]$. If $A$ is a collar neighborhood of $C$, denote by $A_\pm$ the two connected components of $A\setminus C$ corresponding to the connected component of $\bb R\times S^1\setminus C$ asymptotic to $\{\pm\infty\}\times S^1$.

\begin{lemma}\label{lemme_cylindres_feuilletes}
If $(p,R_\alpha)$ is an oriented basis of $T_{x_0}$ and $A$ is small enough, then there exist $x_-$ and $x_+$ in $I_1$ such that \[u_M\left(A_-\right)\subset(x_-,x_0)\times S^1\times S^1 \text{ and } u_M\left(A_+\right)\subset(x_0,x_+)\times S^1\times S^1.\]
Otherwise \[u_M\left(A_-\right)\subset(x_0,x_+)\times S^1\times S^1 \text{ and } u_M\left(A_+\right)\subset(x_-,x_0)\times S^1\times S^1.\]
\end{lemma}

In other words, holomorphic cylinders cross a torus foliated by periodic Reeb orbits in just one direction.

\begin{proof}
Let $C(t)$ be a parametrization of $C$ and $c$ be the projection of $u(C)$ on $T_{x_0}$, $c$ is a smooth curve transverse to $R_\alpha$. If $\big(c'(t_0),R_\alpha(c(t_0))\big)$ is an oriented basis of $T_{x_0}$ for some $t_0$, then $\big(c'(t),R_\alpha(c(t))\big)$ is an oriented basis for all $t$. Thus $(p,R_\alpha)$ is an oriented basis $T_{x_0}$ if and only if $\big(c'(t),R_\alpha(c(t))\big)$ is an oriented basis.

The sets $u_M(A_\pm)$ are connected and therefore contained in $(x_0,x_+)\times S^1\times S^1$ or in $(x_-,x_0)\times S^1\times S^1$. Let $V$ be a normal vector to $C$ at $C(t)$  so that $(V,C'(t))$ is an oriented basis ($V$ points toward $A_+$). Consider $v=\d u_M(C(t))\cdot V$, then
$(v,c'(t),R_\alpha)$ is an oriented basis by positivity of intersection. If $(p,R_\alpha)$ is an oriented basis then the $x$ component of $v$ is positive. Conversely, if $(R_\alpha,p)$ is an oriented basis, the $x$ component of $v$ is negative.
\end{proof}

\section{Contact homology of walled contact structures}\label{section_polynomial}

In this section we prove Theorem \ref{theoreme_cloisonne}. 
The strategy of the proof is to decompose the manifold and its walled contact structure into understandable pieces and to study the contact homology of the pieces. Theorem~\ref{theoreme_cloisonne} states that the cylindrical contact homology of a walled contact structure is the sum of the cylindrical contact homologies of the components of this decomposition.

Let $\pi :M\to S$ be a circle bundle over a closed oriented surface and $\xi$ be a contact structure on $M$ walled by a curve $\Gamma$ that contains no contractible components. To prove Theorem \ref{theoreme_cloisonne} we first construct in Section \ref{section_almost_MB} a contact form ``almost Morse-Bott'' such that
\begin{enumerate}
  \item in a neighborhood of the wall that is a union of thickened tori (one for each connected component of the wall), Reeb orbits foliate all the tori;
  \item elsewhere, the Reeb vector field is tangent to the fibers.
\end{enumerate}
This contact form is not of Morse-Bott type as some spaces $N_T$ have a nonempty boundary. We study the associated periodic Reeb orbits in Section \ref{section_Rpo}.  Then, in Section \ref{section_aMB_perturbation}, we perturb the contact form from Section \ref{section_almost_MB} as in the More-Bott case and control periodic Reeb orbits. We prove the quadratic growth rate of contact homology. In Section \ref{section_control_cylinder}, we prove that there are no holomorphic cylinders between two components of the decomposition using positivity of intersection and end the proof using Morse-Bott theory.
\begin{remark}
Contact structures $\ker(\sin(nx)\d y+\cos(nx)\d z)$ on $T^3=T^2\times S^1$ with coordinates $(x,y,z)$ are walled by the curves $\{(\frac{\pi}{2n}+\frac{k\pi}{n},y),y\in S^1\}$. Theorem~\ref{theoreme_cloisonne} and Corollary~\ref{proposition_homologie_T_3} give the same contact homology.
\end{remark}

\subsection{``Almost Morse-Bott'' contact form}\label{section_almost_MB}

The following proposition results from Giroux's work \cite{Giroux01}.

\begin{proposition}\label{formeparticuliere}
Let $\xi$ be a contact structure on a circle bundle $M$ walled by a nonempty multi-curve $\Gamma$ that contains no contractible component. Then there exist a contact structure $\xi'$ isotopic to $\xi$, a defining contact form $\alpha$ for $\xi'$, and a neighborhood $U$ of the wall with local coordinates $(x,y,z)$ in $(-1,1)\times\cup_{i=1}^n S^1\times S^1$ such that:
\begin{enumerate}
 \item $\pi^{-1}(\Gamma)\simeq\{0\}\times \cup_{i=1}^n S^1 \times S^1$,
 \item $\xi'$ is walled by $\Gamma$,
 \item on a trivialization $S'\times S^1$ of $M\setminus U$, we have $\alpha=\beta+e\d z$ where $\beta$ is a $1$-form on $S'$ and $e=1$ when $\xi$ is positively transverse to the fiber and $e=-1$ when $\xi$ is negatively transverse to the fiber,
 \item $\alpha=f(x)\d y+g(x)\d z$ on $U$ where $f :[-1,1]\to\mathbb R$ is negative and strictly convex, and $g :[-1,1]\to\mathbb R$ has an inflection point at $0$, $g=-1$ on $[-1,-\frac{1}{2}]$,  $g=1$ on $[\frac{1}{2},1]$ and $g$ is increasing in between,
 \item the change of coordinates between $U$ and a neighborhood of $M\setminus U$ is a linear map $(x,y,z)\mapsto (x,y,z+ky)$. 
 \end{enumerate}
\end{proposition}

\begin{figure}[here]
\begin{equation*}
\begin{array}{cccc}
\includegraphics{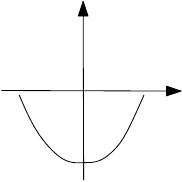} & \includegraphics{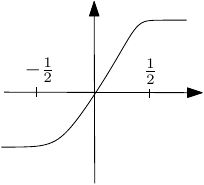}&\includegraphics{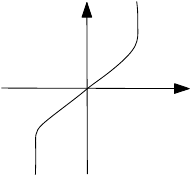}\\
 f &g& \frac{f'}{g'}
\end{array}
\end{equation*}
\caption{Maps $f$ and $g$}
\end{figure}

\begin{remark}\label{remark_R}
On $M\setminus U$, the Reeb vector field is $\pm\frac{\partial}{\partial z}$. On $U$,
\[R_\alpha=\frac{1}{f'(x)g(x)-g'(x)f(x)}\left(\begin{array}{l}
0\\-g'(x)\\f'(x)
\end{array}\right).\] The open set $U$ is a union of thickened tori foliated by Reeb orbits. 
\end{remark}

\begin{proof}
Let $(-1,1)\times \cup_{i=1}^n S^1  \times S^1$ be a chart of a neighborhood $U$ of $\pi^{-1}(\Gamma)\simeq\{0\}\times \cup_{i=1}^n S^1 \times S^1$ with coordinates $(x,y,z)$ so that $\frac{\partial}{\partial x}\in \xi$ and $S^1$ is the fiber. In this chart, any contact form is written
\[\alpha=f(x,y,z)\d y+g(x,y,z)\d z\] where $g(0,y,z)=0$ and $g(x,y,z)\neq0$ for all $x\neq0$. Orient $\Gamma$ so that $\Gamma$ is negatively transverse to $\xi$. Without loss of generality, one can assume $f(0,y,z)=-1$. 

Consider the path of contact forms
\[\alpha_s=\Big(sf(x,y,z)+(1-s)f(x,0,0)\Big)\d y+\Big(sg(x,y,z)+(1-s)g(x,0,0)\Big)\d z\] in a small neighborhood of $\{0\}\times\cup_{i=1}^n S^1  \times S^1$. For all $s\in[0,1]$, $\alpha_s$ is a contact form as $f(0,y,z)=f(0,0,0)$ and $g(0,y,z)=g(0,0,0)$. Using Moser's trick, we can find a vector field $X_s$ near $\{0\}\times \cup_{i=1}^n S^1  \times S^1$ such that
\begin{enumerate}
  \item $X_s(0,y,z)=0$,
  \item  $X_s$ is collinear to $\frac{\partial}{\partial x}$,
  \item $\ker(\phi_s^*\alpha)=\ker(\alpha_s)$ where $\phi_s$ is the flow of $X_s$.
\end{enumerate}
Extend $X_s$ to $M\times[0,1]$ using a cut-off function. Then $\phi_s$ is well defined for all $s\in[0,1]$. The contact structure associated to $\phi_s^*\alpha$ is transverse to the fibers on  $M\setminus\pi^{-1}(\Gamma)$ as $\phi_s^*\alpha=(f\circ\phi_s)\d y+(g\circ\phi_s)\d z$, $(\phi_s)_{|\pi^{-1}(\Gamma)}=\Id$ and $\alpha=\alpha_s$ on $M\setminus U$. Therefore $\xi$ is isotopic to a contact structure with a defining contact form $\alpha$ and a chart
$U'=(-1,1)\times \cup_{i=1}^n S^1  \times S^1$ near $\pi^{-1}(\Gamma)$ where $\alpha=f(x)\d y+g(x)\d z$, $g(0)=0$ and $g(x)\neq 0$ for all $x\neq 0$. By the contact condition, it holds that $g'(0)>0$ and one can assume that $g=-1$ on $[-1,\frac{1}{2}]$ and $g=1$ on $[\frac{1}{2},1]$.

For each connected component of $\Gamma$, choose $f_0$ and $g_0$ so that
\begin{enumerate}
  \item $f_0$ is negative and strictly convex,
  \item $g_0 :[-1,1]\to\mathbb R$ has an inflection point at $0$, $g_0=-1$ on $[-1,-\frac{1}{2}]$,  $g_0=1$ on $[\frac{1}{2},1]$ and $g_0$ is increasing in between,
  \item $f_0=f$ near $x=\pm1$,
  \item $f_0(x)$ (resp. $g_0(x)$) is the same for all connected components and for all $x\in\left[-\frac{3}{4},\frac{3}{4}\right]$.
\end{enumerate}
Write
$f(x)+ig(x)=\rho(x)\exp(i\theta(x))$ and $f_0(x)+ig_0(x)=\rho_0(x)\exp(i\theta_0(x))$. By the contact condition, $\theta$ and $\theta_0$ are decreasing and have the same image as $g(x)=0$ (resp. $g_0(x)=0$) if and only if $x=0$.
By Gray's stability theorem on the path
\[\Big((1-s)\rho(x)+s\rho_0(x)\Big)\exp\Big(i((1-s)\theta(x)+s\theta_0(x))\Big)\]
we obtain an isotopic contact form such that
$\alpha=f_0(x)\d y+g_0(x)\d z$
on $U'$.
Consider $U_i\subset U$ the neighborhood of the component $\Gamma_i$ with coordinates $(-\frac{1}{2},\frac{1}{2})\times S^1\times S^1$.
Let $V$ be a neighborhood of a connected component of
$M\setminus \cup_{i=1}^n U_i $. As $\Gamma\neq \emptyset$ and $S$ is connected, $V$ is a manifold with boundary and the circle bundle is trivial. Let $S'\times S^1$ be a trivialization such that the change of coordinates between $V$ and $(-1,1)\times \cup_{i=1}^n S^1\times S^1$ is linear (i.e. $(x,y,z)\mapsto (x,y,z+ky)$) in polar coordinates near the boundary. Therefore $\alpha=\beta+e\d z$ near $\partial V$. On $V$, $\alpha=\beta_z+h\d z$ and $h\neq 0$, so one can assume $h=e$. By use of Gray's theorem on the path
$\alpha_s=s\beta_z(x)+(1-s)\beta_0(x)+e\d z$ we obtain the desired contact form.
\end{proof}

We are already in position to prove the third part of Theorem~\ref{theoreme_cloisonne}.

\begin{proof}[Proof of Theorem \ref{theoreme_cloisonne}, part (\ref{C3})]
Let $\eta$ be a loop that neither is in the homotopy class of a multiple of the fiber nor projects to a multiple of a curve $\Gamma_i$. Then, the Reeb vector field associated to the contact form $\alpha$ given in Proposition \ref{formeparticuliere} does no have periodic orbit in the homotopy class $[\eta]$. Therefore, there exists a contact form~$\alpha$ such that $HC_*^{[\eta]}(M,\alpha)=0$.
\end{proof}

\subsection{$R_\alpha$ periodic orbits}\label{section_Rpo}
We first investigate the growth rate of Morse-Bott tori.
Let $\alpha$ be a contact form given in Proposition \ref{formeparticuliere}.
Consider $U_i\subset U$ the neighborhood of the component $\Gamma_i$ with coordinates $(-\frac{1}{2},\frac{1}{2})\times S^1\times S^1$. Let $W=M\setminus \cup_{i=1}^n U_i$ (see Figure \ref{figure_U}). 
\begin{figure}[here]
\includegraphics{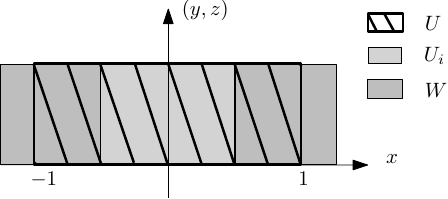}
\caption{Sets $U$, $U_i$ and $W$}\label{figure_U}
\end{figure}
On $W$, all the fibers are periodic orbits of period $1$. On $U_i$, $\alpha=f(x)\d y+g(x)\d z$ and the Reeb vector field is given by
\[R_\alpha=\frac{1}{f'g-fg'}\left(\begin{array}{c} 
 0 \\ 
 -g'\\ 
f' 
\end{array}\right).\] 
There are two cases:  $f'(x)\neq 0$ and $g'(x)\neq 0$.
If $f'(x)\neq 0$ and $-\frac{g'(x)}{f'(x)}=\frac{p}{q}$ with $\gcd(p,q)=1$, the period of the periodic Reeb orbits in $T_x=\{x\}\times S^1\times S^1$ is
\[T=\left\arrowvert\frac{(f'g-fg')q}{f'}\right\arrowvert=\left\arrowvert qg+fp\right\arrowvert.\] 
If $g'(x)\neq 0$ and $-\frac{f'(x)}{g'(x)}=\frac{q}{p}$ with  $\gcd(p,q)=1$, the period of the periodic Reeb orbits in $T_x$ is \[T=\left\arrowvert\frac{(f'g-fg')p}{g'}\right\arrowvert=\left\arrowvert qg+fp\right\arrowvert.\]
In what follows we will assume $q\geq 0$.

Recall that we write $\sigma(\alpha)$ the action spectrum, $N_L=\{p\in M, \phi_L(p)=p\}$ and $S_L=N_L/\raisebox{-0.4ex}{$S^1$}$ for all $L\in\sigma(\alpha)$.
\begin{lemma}\label{lemme_croissance_sigma}
The set $\sigma(\alpha)$ is discrete and $\#(\sigma(\alpha)\cap[0,L])$ exhibits exactly quadratic growth with $L$. 
\end{lemma}

\begin{proof} 
We first show that it has a quadratic upper bound.
There exist $A>0$ and intervals $I_1$ and $I_2$ such that
\begin{enumerate}
  \item $I_1\cup I_2=(-\frac{1}{2},\frac{1}{2})$,
  \item $\left\vert \frac{1}{g'}\right\vert<A$ on $I_1$ and $\left\vert \frac{1}{f'}\right\vert<A$ on $I_2$,
  \item $\frac{1}{A}<f'g-fg'<A$, $\vert f'\vert<A$ and $\vert g'\vert <A$.
\end{enumerate}
Let $L>0$. There are at most $3A^6 L^2$ rational numbers $\frac{q}{p}$ such that $\left\vert\frac{(f'g-fg')p}{g'(x)}\right\vert<L$ for some $x\in I_1$ with $-\frac{f'(x)}{g'(x)}=\frac{q}{p}$. As $x\mapsto-\frac{f'(x)}{g'(x)}$ is increasing, for all rational numbers $\frac{q}{p}$ there is one $x$ such that $\frac{f'(x)}{-g'(x)}=\frac{q}{p}$. Therefore the growth rate is at most quadratic.

We now show that the growth rate is also at least quadratic: consider $B>0$ and $p$, $q$ such that $p^2+q^2\leq B$. Then, there exists $x$ such that $-\frac{f'(x)}{g'(x)}=\frac{q}{p}$. The associated torus is foliated by periodic Reeb orbits of period smaller than $A^2B$.
\end{proof}

\subsection{Periodic orbits of the perturbed contact form}\label{section_aMB_perturbation}
Let $\alpha$ be a contact form given in Proposition \ref{formeparticuliere}. We use the notation from Section~\ref{section_Rpo}.
Similar to the method we use in the Morse-Bott case, we perturb the degenerate contact form.
Fix $L>0$. For all $L'\leq L$, fix a Morse function $f_{L'}$ on $S_{L'}$ and extend it to $N_{L'}$ so that $\d f_{L'}(R_\alpha)=0$. We assume that for $L'=1$, $f_1$ does not depend on $y$ and $e\frac{\partial f_1}{\partial x}>0$ in the cylindrical coordinates $(x,y,z)$ near $\partial W$.
Extend $f_{L'}$ to $M$ by using cut-off functions in a standard way. Let $\overline{f}_L$ denote the sum of $f_{L'}$ for all $L'\leq L$ and perturb the contact form with $\overline{f}_L$. More precisely, let
\[\alpha_{\lambda,L}=(1+\lambda\overline{f}_L)\alpha.\]
Note that the diameter in the $x$-coordinate of connected components of $\dom(\overline{f}_L)$ that do not contain $W$ tends to $0$ as $L\to\infty$. Additionally, the flow of $R_{\alpha_{L,\lambda}}$ preserves $\dom (\overline{f}_L)$ and $M\setminus\dom (\overline{f}_L)$.

\begin{lemma}\label{perturbation_morse}
For all $L>0$, there exists $\lambda(L)>0$ such that for all $0<\lambda\leq\lambda(L)$, the periodic orbits of $\alpha_{\lambda,L}$ with period smaller than $L$ correspond to critical points of $f_{L'}$ on $S_{L'}$ for $L'\in\sigma(\alpha)\cap[0,L]$. These periodic orbits are non-degenerate.
\end{lemma}

\begin{proof}
Outside a neighborhood of $\partial W$, Morse-Bott theory applies directly (see \cite[Lemma 2.3]{Bourgeois02}). In a neighborhood of $\partial W$, in the trivializing chart of $W$ with coordinates  $(x,y,z)$, the contact form is written
\[\alpha=(f(x)+kg(x))\d y+g(x)\d z=f_W(x)\d y+g(x)\d z\] 
as the change of coordinates is linear (Proposition \ref{formeparticuliere}).
As $\overline{f}_L$ only depends on $x$, we have
\[R_{\alpha_{L,\lambda}}=\frac{1}{(f_W'g-f_Wg')(1+\lambda \overline{f}_L)^2}
\left(\begin{array}{c}
0\\
 -g'(1+\lambda \overline{f}_L)-\lambda g \overline{f}_L'\\
f_W'(1+\lambda \overline{f}_L)+\lambda f_W \overline{f}_L'
\end{array} \right).\]
In a small neighborhood of $\partial W$ and for $\lambda$ small enough, the $y$ coordinate of the Reeb vector field is as small as desired and does not vanish. Therefore there is no periodic Reeb orbit with period smaller than $L$.
\end{proof}

\begin{lemma}\label{OP_et_points_critiques_cas_cloisonné}\label{noncontractile}
Let $\eta$ be a loop that is a multiple of the fiber or projects to a multiple of a connected component of $\Gamma$. Then, there exist $L_0>0$ and $L\mapsto \lambda(L)$ positive and decreasing such that for all $L\geq L_0$ and $\lambda\leq\lambda(L)$
\begin{enumerate}
  \item the periodic Reeb orbits of $\alpha_{L,\lambda}$ homotopic to $\eta$ have a period smaller than~$L$,
  \item $\alpha_{L,\lambda}$ is hypertight.
\end{enumerate}
In addition, there exist arbitrarily small non-degenerate and hypertight perturbations of $\alpha_{L,\lambda}$.
\end{lemma}

\begin{proof}
Let $R_y$ and $R_z$ denote the $y$ and $z$-coordinates of the Reeb vector field.
Let $W_1,\dots, W_m$ be the connected components of $W$. Consider 
\[\bigcup_{i=1}^m W'_i\cup\bigcup_{j=1}^n U'_j\]
an open covering of $M$ such that $W'_i\cap \pi^{-1}(\Gamma)=\emptyset$ for all $i=1\dots m$ and $U'_j\cap W=\emptyset$ for all $j=1\dots n$ (see Figure \ref{figure_U_1}). 
\begin{figure}[here]
\includegraphics{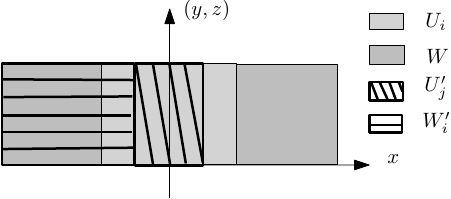}
\caption{Sets $U'_j$ and $W'_i$}\label{figure_U_1}
\end{figure}
There exists $\epsilon>0$ such that, in the trivialization of $W'_i$ induced by Proposition \ref{formeparticuliere}, $\vert R_z\vert>\epsilon$ and, in the trivialization of $U'_j$ induced by Proposition \ref{formeparticuliere}, $\vert R_y\vert>\epsilon$. If there exists a loop $\eta'$ in $W_i$ (resp. $U_j$) such that $[\eta']=[\eta]$, let $k_i$ (resp. $k'_j$) denote the multiplicity of the fiber (resp. $\Gamma_j$) in the decomposition of $\eta'$ in the associated trivialization. Else, set $k_i=1$ (resp. $k'_j=1$).
Consider $L_0>0$ such that
\begin{enumerate}
 \item\label{cond'_1} $L_0>\displaystyle{\max_{1\leq i\leq m, 1\leq j\leq n}}(\{|k_i|,|k_j'|\})\cdot\frac{1}{\epsilon}$,
 \item\label{cond'_2} periodic orbits of $R_\alpha$ homotopic to $\eta$ have period smaller than $L_0$,
 \item\label{cond'_3} for all $L''\geq L_0$ , the connected components of $\dom(f_{L''})$ are contained either in a $U'_j$ or in a $W'_i$.
\end{enumerate}
Let $L\geq L_0$.
By Lemma \ref{perturbation_morse}, there exists $\lambda(L)$ such that for all $\lambda\leq\lambda(L)$, all periodic Reeb orbits of $\alpha_{L,\lambda}$ with period smaller than $L$ are non-degenerate and such that $\epsilon$ is a lower bound for the $z$-component of $R_{\alpha_{L,\lambda}}$ in $W'_i$ and for the $y$-component of $R_{\alpha_{L,\lambda}}$ in $U'_j$.

Let $\gamma$ be a $\alpha_{L,\lambda}$-periodic Reeb orbit with period greater than $L$. Then, either $\gamma\subset\dom(f_L)$ or $\gamma\subset M\setminus\dom(f_L)$. If $\gamma\subset\big(M\setminus\dom(f_L)\big)$ then $\gamma$ is not homotopic to $\eta$ by condition~\eqref{cond'_2}. If $\gamma\subset\dom(f_L)$, by condition~\eqref{cond'_3}, either $\gamma\subset \big( \dom(f_L)\cap W'_i\big)$ or $\gamma\subset ( \dom(f_L)\cap U'_j\big)$. If $\gamma\subset \big( \dom(f_L)\cap W'_i\big)$, then $\gamma$ covers the fiber at least $\pm\epsilon L$ times and hence $\gamma$ covers the fiber at least $|k_i|+1$ or  $-|k_i|-1$ times by condition~\eqref{cond'_1}. If $\gamma\subset ( \dom(f_L)\cap U'_j\big)$ then it covers $\Gamma_j$ at least $|k'_j|+1$ or  $-|k'_j|-1$ times. Consequently, $\gamma$ is not homotopic to $\eta$ and is non-contractible as $\Gamma_j$ is not contractible and the fiber is not a torsion element. By Lemma \ref{perturbation_morse}, $\alpha_{L,\lambda}$ is hypertight.

We now focus on the existence of non-degenerate, hypertight perturbations of $\alpha_{L,\lambda}$. We may assume that the boundaries of $U'_j$ and $W'_i$ are tori $x=\text{cst}$ with dense Reeb orbits. Choose a small non-degenerate perturbation $\alpha'$ of $\alpha_{L,\lambda}$ that preserves the 
boundaries of $U'_j$ and $W'_i$ and such that $\epsilon$ is a lower bound for the $z$-component of $R_{\alpha'}$ in $W'_i$ and for the $y$-component of $R_{\alpha'}$ in $U'_j$. If $\gamma$ is a periodic Reeb orbit, then $\gamma$ is contained either in a $U'_j$ or in a $W'_i$. As in the previous
paragraph, $\gamma$ is non-contractible.
\end{proof}

\begin{lemma}
Under Hypothesis H, the growth rate of contact homology is (at most) quadratic.
\end{lemma}

\begin{proof}
Let $\alpha'$ be a non-degenerate and hypertight contact form (given for instance by Lemma \ref{noncontractile}). Let $\alpha_{L_i,\lambda_i}$ be a sequence of contact forms with $L_i\to\infty$ such that $L_i\notin\sigma(\alpha)$ and $\lambda_i\leq\lambda(L_i)$ for all $i\in\bb N^*$. Perturb $\alpha_{L_i,\lambda_i}$ to obtain a non-degenerate hypertight form $\alpha'_{L_i,\lambda_i}$ (Lemma \ref{noncontractile}). For $\lambda_i$ small enough and for small perturbations, the periodic Reeb orbits of $\alpha'_{L_i,\lambda_i}$ with period smaller than $L_i$ are in bijection with the periodic Reeb orbits of $\alpha_{L_i,\lambda_i}$ with period smaller than $L_i$ and the difference between their period and the period of the associated $R_\alpha$ periodic orbits is bounded by $\frac{1}{2}$. Thus, there exists $C>0$ such that for all $i\in \bb N^*$ and for all $L\leq L_i$ 
\[\#C_{\leq L}^\t{cyl}(M,\alpha'_{L_i,\lambda_i})\leq C\#\left(\sigma(\alpha)\cap[0,L+1]\right).\]
Let  $\alpha'=f_{L_i,\lambda_i}\alpha'_{L_i,\lambda_i}$. As the $\alpha'_{L_i,\lambda_i}$ are perturbations of $\alpha$, there exists $D>0$ such that
\[\frac{1}{D}<\displaystyle{\sup_{p\in M}}\left\{f_{L_i,\lambda_i}(p),\frac{1}{f_{L_i,\lambda_i}(p)}\right\}<D.\]
By invariance of cylindrical contact homology (Corollary \ref{proposition_invariance_croissance_cylindrique}) and by \cite[Section 10]{ColinHonda08}, there exists $C(D)$ such that, for all $L>0$ and for all $i$, \[\rank(\psi_L)\leq\rank(\psi^i_{C(D)L})\] where $\psi^i_L : HC_{\leq L}^\t{cyl}(M,\alpha'_{L_i,\lambda_i})\to HC^\t{cyl}(M,\alpha'_{L_i,\lambda_i})$ and $\psi_L : HC_{\leq L}^\t{cyl}(M,\alpha')\to HC^\t{cyl}(M,\alpha')$ are the maps defining the direct limit. Hence, \[\rank(\psi_L)\leq C\#\left(\sigma(\alpha)\cap[0,C(D)L+1]\right)\] and $\rank(\psi_L)$ exhibits a quadratic growth.
\end{proof}

\subsection{Holomorphic cylinders and Morse-Bott theory}\label{section_control_cylinder}
Let $\eta$ be a loop that is a multiple of the fiber or projects on a multiple of a connected component of $\Gamma$.
By Lemma \ref{OP_et_points_critiques_cas_cloisonné}, there exist $L>0$ and $\lambda>0$ such that all the $R_{\alpha_{L,\lambda}}$-periodic orbits homotopic to $\eta$ are non-degenerate, associated to a critical point of $f_{L'}$ for  $L'\leq L$ and have period smaller than $L$.
Consider $V_j=W_j\cup \bigcup_{k\in K_j} U_k$ where $K_j=\{k,U_k \text{ is adjacent to } W_j \}$. Then $V_j$ is a trivial circle bundle. Extend the trivialization from Proposition \ref{formeparticuliere} in $V_j\simeq S'_j\times S^1$. In these coordinates, $\alpha=\left(f(x)+mg(x)\right)\d y+g(x)\d z$ and the Reeb vector field is positively collinear to
\[\left(\begin{array}{c}
 0 \\
 -g'\\
f'+mg'
\end{array}\right).\]
Note that the $y$-coordinate is negative in $V_j\setminus W_j$.

\begin{lemma}\label{controle_cylindre_cloisonne_I}
Let $u: \mathbb R\times S^1\to \mathbb R\times M$ be a holomorphic cylinder negatively asymptotic to $\gamma\in W_j$. Then $u_M(\mathbb R\times S^1)\subset W_j$.
\end{lemma}

\begin{proof}
We prove the lemma by contradiction. Assume $u_M(\mathbb R\times S^1)\cap U_l\neq\emptyset$ for $l\in K_j$, then there exists an open interval $I$ such that, in $I\times S^1\times S^1\subset U_l$,
\begin{enumerate}
  \item $\alpha=f_1(x)\d\theta+g_1(x)\d z$ in the trivialization of $V_j$;
  \item $u_M(\mathbb R\times S^1)\cap\{x\}\times S^1\times S^1 \neq\emptyset$ for all $x\in I$.
\end{enumerate}
By Lemma \ref{lemme_u_-1}, for all $x_0\in I$,
\[u^{-1}\left(u(\mathbb R\times S^1)\cap\mathbb R\times T_{x_0}\right)\]
is a finite union of smooth circles homotopic $\{*\}\times S^1$. 
For all $l\in K_j$, choose $x_0\in I$ and cut $\mathbb R\times S^1$ along the associated circles.
Choose the connected component asymptotic to $-\infty\times S^1$. Let $C$ denote the oriented boundary of this component and choose a collar neighborhood $A=\overline{A_+}\cup\overline{A_-}$ of $C$ as in Lemma \ref{lemme_cylindres_feuilletes}: $A_\pm$ are open annuli in the connected component of $\mathbb R\times S^1\setminus C$ asymptotic to $\pm\infty\times S^1$. Let $W_+$ (resp $W_-$) denote the union of the connected components of $W$ such that the Reeb vector field is positively (resp negatively) tangent to the fiber.
 
If $\gamma\subset W_+$, the line in $T_{x_0}$ tangent to $p=(0,1)$ is in the homotopy class of $\gamma$. Hence $(p,R_\alpha)$ is an oriented basis. By Lemma \ref{lemme_cylindres_feuilletes},
$u_M\left(A_-\right)\subset(x_-,x_0)\times S^1\times S^1$.
\begin{center}
 \begin{tikzpicture}
  \draw [->](0,0) -- (4,0);
  \draw (2,-0.25) -- (2,0.6);
  \draw node at (0,0.4) {$W_-$}; 
  \draw node at (4,0.4) {$W_+$};
  \draw node at (2,-0.4) {$x_0$}; 
  \draw node at (1.3,.4) {$u_M(A_-)$}; 
  \draw node at (2.7,.4) {$u_M(A_+)$}; 
  \draw node at (4,-0.4) {$x$};
\end{tikzpicture}
\end{center}
Yet $u_M\left(A_-\right)\subset(x_0,x_+)\times S^1\times S^1$ as $u$ is negatively asymptotic to $\gamma$. This leads to a contradiction.

If $\gamma\subset W_-$, the line tangent to $p=(0,-1)$ in $T_{x_0}$ is in the homotopy class of $\gamma$ and $(p,R_\alpha)$ is not an oriented basis. Thus
$u_M\left(A_-\right)\subset(x_0,x_+)\times S^1\times S^1$. 
\begin{center}
 \begin{tikzpicture}
  \draw [->](0,0) -- (4,0);
  \draw (2,-0.25) -- (2,0.6);
  \draw node at (0,0.4) {$W_-$}; 
  \draw node at (4,0.4) {$W_+$};
  \draw node at (2,-0.4) {$x_0$}; 
  \draw node at (1.3,.4) {$u_M(A_+)$}; 
  \draw node at (2.7,.4) {$u_M(A_-)$}; 
  \draw node at (4,-0.4) {$x$};
\end{tikzpicture}
\end{center}
This leads to a contradiction as $u$ is negatively asymptotic to $\gamma$.
\end{proof}

\begin{lemma}\label{controle_cylindre_cloisonne_II}
Let $u : \mathbb R\times S^1\to \mathbb R\times M$ a holomorphic cylinder negatively asymptotic to $\gamma\in U_j$. Then $u_M(\mathbb R\times S^1)\subset U_j$ and $u_M(\mathbb R\times S^1)\subset\dom(\overline{f_L})$
\end{lemma}

\begin{proof}
Consider $x_0$ such that $\gamma\in T_{x_0}$ in the trivialization $(-\frac{1}{2},\frac{1}{2})\times S^1\times S^1$ of $U_j$.  We prove the lemma by contradiction. Thus there exists an open interval $I\subset (-\frac{1}{2},\frac{1}{2})$ such that $\alpha=f(x)\d y+g(x)\d z$ and $f(\mathbb R\times S^1)\cap\{x\}\times S^1\times S^1 \neq\emptyset$ for all $x\in I$. By Lemma \ref{lemme_u_-1}, for all $x_1\in I$,
\[u^{-1}\left(u(\mathbb R\times S^1)\cap\mathbb R\times T_{x_1}\right)\]
is a finite union of smooth circles homotopic to $\{*\}\times S^1$. Cut $\mathbb R\times S^1$ along these circles and denote by $C$ the oriented boundary of the component asymptotic to $-\infty\times S^1$. 
Let $p$ be a vector tangent to $T_{x_0}$ so that the straight line in $T_{x_0}$ directed by $p$ is in the homotopy class $[\gamma]$. Consider a collar neighborhood $A=\overline{A_+}\cup\overline{A_-}$ of $C$ as in Lemma \ref{lemme_cylindres_feuilletes}: $A_\pm$ are open annuli in the connected component of $\mathbb R\times S^1\setminus C$ asymptotic to $\pm\infty\times S^1$.
 
If $x_1>x_0$ then $(p,R_\alpha)$ is not an oriented basis ($f'$ is increasing) and $u_M\left(A_-\right)\subset(x_1,x_+)\times S^1\times S^1$ by Lemma \ref{lemme_cylindres_feuilletes}. If $x_1<x_0$ then $(p,R_\alpha)$ is an oriented basis and $u_M\left(A_-\right)\subset(x_-,x_1)\times S^1\times S^1$.
\begin{center}
 \begin{tikzpicture}
  \draw [->](0,0) -- (8,0);
  \draw (2,-0.25) -- (2,0.6); 
  \draw (4,-0.25) -- (4,0.25);
  \draw (6,-0.25) -- (6,0.6);
  \draw node at (8,-0.4) {$x$};

  \draw node at (2,-0.4) {$x_1$}; 
  \draw node at (1.3,.4) {$u_M(A_-)$}; 
  \draw node at (2.7,.4) {$u_M(A_+)$}; 

  \draw node at (6,-0.4) {$x_1$}; 
  \draw node at (5.3,.4) {$u_M(A_+)$}; 
  \draw node at (6.7,.4) {$u_M(A_-)$}; 

  \draw node at (4,-0.4) {$x_0$};
\end{tikzpicture}
\end{center}
This leads to a contradiction as $u$ is negatively asymptotic to $\gamma$.
\end{proof}

\begin{lemma}\label{prolongementW_jU_i}
For all $j=1\dots m$, there exists a contact closed manifold without boundary $(\tilde{W}_j,\tilde{\alpha})$ extending $(W_j,\alpha)$ such that $\tilde{\alpha}$ is of Morse-Bott type. For all $i=1\dots n$, there exists a contact closed manifold without boundary $(\tilde{U}_i,\tilde{\alpha})$ extending $(U_i,\alpha)$ such that $\tilde{\alpha}$ is of Morse-Bott type. 
\end{lemma}

\begin{proof}
In the trivialization $W_j\simeq S_j\times S^1$, the contact form is $\alpha=\beta+e\d z$ and, near $\partial W_j$, there exist coordinates $(x,y,z)\in [0,1]\times S^1\times S^1$ such that $\{1\}\times S^1\times S^1\subset \partial W_j$ and $\alpha=f(x)\d y+e\d z$. 
Let $S'$ be an oriented compact surface such that $\partial S'$ and $\partial S_j$ have the same number of connected components. Choose a pairing between these components and glue a neighborhood of each component of $\partial W_j$ to a neighborhood of the associated component of $\partial S'\times S^1$ with the diffeomorphism $\phi : (x,y,z)\mapsto (x,y,z+ky)$ where $k\in\mathbb Z$. Let $\tilde{W}_j$ denote the resulting manifold. Near $\partial S'\times S^1$,\[\phi^*\alpha=\big(f(x)+ke\big)\d y+e\d z=\tilde{\beta}_k+e\d z.\] For each component, choose $k$ so that $e\tilde{\beta}_k$ is positive on $\partial S'$. There exists a $1$-form $\beta'$ on $S'$ such that $e\d\beta'>0$ and $\tilde{\beta}_k=\beta'$ near the boundary. The contact form $\beta'+e \d z$ extends $\phi^*\alpha$ and the induced form $\tilde{\alpha}$ on $\tilde{W}_j$ is of Morse-Bott type.

On $U_j=A\times S^1$, the contact form is written $\alpha=f(x)\d y+g(x)\d z$. Extend $f$ and $g$ to maps $\tilde f$ and $\tilde g$ on $S^1$ so that $\tilde\alpha=\tilde f(x)\d y+\tilde g(x)\d z$ is a contact form on $T^3$. The form $\tilde\alpha$ is of Morse-Bott type.
\end{proof}

\begin{proof}[Proof of Theorem \ref{theoreme_cloisonne}]
It remains to compute contact homology for $\eta$ satisfying condition \eqref{C1} or \eqref{C2}. By Lemma \ref{OP_et_points_critiques_cas_cloisonné}, there exist $L>0$ and $\lambda>0$ such that all the $R_{\alpha_{L,\lambda}}$-periodic orbits homotopic to $\eta$ have a period smaller than $L$, are non-degenerate and associated to a critical point of $f_{L'}$ for $L'\leq L$. 

Extend $\overline{f}_L$ to the contact manifolds $\tilde{W}_j$ and $\tilde{U}_i$ (Lemma \ref{prolongementW_jU_i}) to get a Morse-Bott perturbation. Let $(\lambda_l)$ be a decreasing sequence such that $\lambda_l\in(0,\lambda]$ and $\lim_{l\to\infty}\lambda_l=0$. Choose almost complex structures $J_{\lambda_l}$ adapted to $(M,\alpha_{L,\lambda_l})$ and $\tilde{J}_{\lambda_l}$ adapted to the union of $(\tilde{W}_j,\tilde{\alpha}_{L,\lambda_l})$ for $j=1\dots m$ and $(\tilde{U}_i,\tilde{\alpha}_{L,\lambda_l})$ for $i=1\dots n$ so that
\begin{enumerate}
  \item $J_{\lambda_l}=\tilde{J}_{\lambda_l}$ on $\bb R\times W_j$ and $\bb R\times U_i$;
  \item $J_{\lambda_l}$ and $\tilde{J}_{\lambda_l}$ are $S^1$-invariant on $\bb R\times N_{L'}$ for all $L'\leq L$;
  \item $(f_{L'},g_{L'})$ (resp $(\tilde f_{L'},\tilde g_{L'})$) is Morse-Smale on $S_{L'}$ where  $g_{L'}$ (resp $\tilde g_{L'}$) is the metric induced by $J_{\lambda_l}$ (resp $\tilde{J}_{\lambda_l}$) for all $L'\leq L$.
\end{enumerate}

By Theorems \ref{theoreme_MB_S_1_invariant_modules} and \ref{theoreme_MB_T_3_invariant_modules}, for all $j=1\dots m$ (resp $i=1\dots n$) and for $l$ big enough, $\tilde{J}_{\lambda_l}$-holomorphic cylinders asymptotic periodic Reeb orbits in $W_j$ (resp. $U_i$) are contained in $W_j$ (resp. $U_i$) as gradient lines between two points in $S_j$ are contained in $S_j$. 

By Lemmas \ref{controle_cylindre_cloisonne_I} and \ref{controle_cylindre_cloisonne_II}, for all $j=1\dots m$ (resp $i=1\dots n$) and for $l$ big enough, $J_{\lambda_l}$-holomorphic cylinders asymptotic periodic Reeb orbits in $W_j$ (resp. $U_i$) are contained in $W_j$ (resp. $U_i$). Therefore the differential of contact homology is well defined and can be identified with the differential in the Morse-Bott case and thus with the differential in Morse homology. Hence
\begin{enumerate}
 \item if $[\eta]=[\text{fiber}]^k$ with $\pm k>0$, \[HC_*^{[\eta]}(M,\alpha,\mathbb Q)=\displaystyle{\bigoplus_{W_j\subset W_\pm}} H^{\t{Morse}}_*(W_j, (f_1,g_1),\mathbb Q)\]
 \item if $[\pi(\eta)]=[\Gamma_j]^{k'}$ with $k'\neq 0$, \[HC_*^{[\eta]}(M,\alpha,\mathbb Q)=\displaystyle{\bigoplus_{i\in I}} H^{\t{Morse}}_*(S^1, (f_L,g_L),\mathbb Q)\]
\end{enumerate}
where $I=\{i,[\pi(\Gamma_i)]=[\pi(\Gamma_j)]\}$ (we do not consider the graduation in the identifications).
\end{proof}
In their study of algebraic torsion in SFT, Latschev and Wendl \cite{LatschevWendl10} used similar methods to understand holomorphic curves.

\section{Hyperbolicity and exponential growth rate}\label{section_hyperbolic}

In this section we prove Theorem \ref{theoreme_hyperbolique}. This result hinges on the exponential growth of contact homology for a specific family of contact structures (Theorem \ref{theoreme_hyperbolique_homologie}). The invariance of contact homology leads to the exponential growth of $N_L(\alpha)$ for all non-degenerate contact forms. For a general non-degenerate contact form, the proof depends on Hypothesis H.

Let $M$ be a $3$-manifold which can be cut along a nonempty family of incompressible tori $T_1,\dots T_N$ into irreducible manifolds including a hyperbolic component that fibers on the circle. We construct contact forms on each irreducible components and add torsion near the incompressible tori $T_k$ for $k=1\dots N$ (Section \ref{section_equation_pA}). We compute the growth rate of contact homology by controlling the holomorphic cylinders that intersect the tori $T_k$ for $k=1\dots N$ (Section \ref{section_HC_pA}). The study of periodic orbits and contact homology in the hyperbolic component hinges on properties of periodic points of pseudo-Anosov automorphisms recalled in Section \ref{proprietes_pA}.

\subsection{Periodic points of pseudo-Anosov automorphisms}\label{proprietes_pA}
See \cite{CassonBleiler88,FLP91} for a precise definition of  pseudo-Anosov automorphisms. Here, we will only use the properties of pseudo-Anosov automorphisms described in Theorems \ref{theoreme_croissance_pA}, \ref{theoreme_Nielsen_pA}, \ref{theoreme_invariance_Nielsen} and \ref{theoreme_pA_boundary}.
Let $S$ be a compact orientable surface. An automorphism $\psi : S\to S$ is said to be \emph{pseudo-Anosov} if there exist two measured foliations $(\mathscr F^s,\mu^s)$ and $(\mathscr F^u,\mu^u)$ such that $\psi(\mathscr F^u,\mu^u)=(\mathscr F^u,\lambda^{-1}\mu^u)$ and $\psi(\mathscr F^s,\mu^s)=(\mathscr F^s,\lambda\mu^s)$ for a positive real number $\lambda$. In this section, we assume that $S$ has no boundary.

\begin{theorem}[{see \cite[Theorem 11.1]{ColinHonda08}}]\label{theoreme_croissance_pA}
The number of simple $k$-periodic points of a pseudo-Anosov automorphism on $S$ exhibits an exponential growth with $k$.
\end{theorem}  

This theorem follows from the construction of a Markov partition on $S$ (see \cite[Section 11]{ColinHonda08}).
Nielsen classes are used to transfer properties of periodic points of a pseudo-Anosov map to properties of periodic points of homotopic diffeomorphisms.
\begin{definition} 
Let $h : S\to S$ be an automorphism. Two fixed points $x$ and $y$ are \emph{in the same Nielsen class} if there exists a continuous map $\delta :[0,1]\to S$ such that $\delta(0)=x$, $\delta(1)=y$ and $h(\delta)$ is homotopic to $\delta$. 
Let $h_t : S\to S, t\in[0,1]$ be a homotopy of automorphism of $S$. Fixed points $x_0$ of $h_0$ and $x_1$ of $h_1$ are \emph{in the same Nielsen class} if there exists a continuous map $\delta :[0,1]\to S$ such that $\delta(0)=x$, $\delta(1)=y$ and $ h_\cdot(\delta(\cdot))$ is homotopic to $\delta$. 
\end{definition} 
One can refer to \cite{Felshtyn00} for more information on Nielsen classes. These definitions extend naturally to periodic points. Two periodic points are in the same Nielsen class of a diffeomorphism $h$ if and only if the induced periodic orbits of the vertical vector field in the mapping torus $(S\times \mathbb R)/h$ are homotopic. Nielsen classes of a pseudo-Anosov automorphisms are very special.

\begin{theorem}[Thurston, Handel {\cite[Lemma 2.1]{Handel85}}]\label{theoreme_Nielsen_pA} 
All the periodic points of a pseudo-Anosov automorphism on $S$ are in different Nielsen classes.
\end{theorem} 
A $k$-periodic point $x$ of $h : S\to S$ is \emph{non-degenerate} if $1$ is not an eigenvalue of $\d h^k(x)$.
For a non-degenerate $k$-periodic point, let $\epsilon_{h^k}(x)$ denote the sign of $\det (\d h^k(x)-\Id)$. If all the periodic points in a given Nielsen class $c$ are non-degenerate, we define
\[\Lambda_{h^k}(c)=\sum_{x\in c}\epsilon_{h^k}(x).\] 

\begin{theorem}[\cite{Jiang83}]\label{theoreme_invariance_Nielsen} 
Let $h_0$ and $h_1$ be two homotopic automorphisms of $S$. Let $x_0$ and $x_1$ be two periodic points of $h_0$ and $h_1$ in the same Nielsen class. If the Nielsen classes $c_0$ of $x_0$ (for $h_0$) and $c_1$ of $x_1$ (for $h_1$) contain only non-degenerate points then $\Lambda_{h_0^k}(c_0)= \Lambda_{h_1^k}(c_1)$. 
\end{theorem} 

\begin{theorem}[{see \cite[Section 11.1]{ColinHonda08}}]\label{theoreme_pA_boundary}
Let $S_1$ be a compact surface with boundary obtained from $S$ after removing a finite number of disjoint open disks. Let $h:S_1\to S_1$ be an automorphism such that $h=\Id$ in a neighborhood of $\partial S_1$ and homotopic to a pseudo-Anosov automorphism $\psi$. Extend $h$ to $S$ by the identity and let $\hat h$ denote the resulting automorphism. Then, there exist a branched cover $\hat{S}$ of $S$ and a pseudo-Anosov map $\hat\psi$ homotopic to $\hat h$ such that the projection of $\hat\psi$ is $\psi$. 
\end{theorem}

\subsection{Contact forms on $M$}\label{section_equation_pA}

\subsubsection{In the hyperbolic component} 
Let $M_0$ be the hyperbolic component, then $M_0$ is the mapping torus $(S\times \mathbb R)/h$ where
\begin{enumerate}
  \item $S$ is a compact oriented surface with boundary,
  \item $h : S\to S$ is a diffeomorphism homotopic to a pseudo-Anosov map,
  \item $h=\Id$ in a neighborhood of $\partial S$.
\end{enumerate}
We use the usual construction on a contact structure on a mapping torus. Choose cylindrical coordinates $(r,\theta)$ in a neighborhood of $\partial S$ so that $\frac{\partial}{\partial \theta}$ is positively tangent to $\partial S$. Let $\beta$ be a $1$-form on $S$ such that $\d \beta>0$ and, near $\partial S$, $\beta=b(r)\d \theta$ with $b>0$ and $b'>0$. Let $F:[0,1]\to [0,1]$ be a smooth non-decreasing function such that $F=0$ near $0$ and $F=1$ near $1$. On $S\times[0,1]$ consider the contact form 
\[\alpha=(1-F(t))\beta+F(t)h^*\beta+\d t \] 
where $t$ is the coordinate on $[0,1]$. This contact form induces a contact form on $M_0$. The associated contact structure is universally tight.
\begin{lemma} 
The Reeb vector field is positively transverse to $S\times\{*\}$ and the first return map on $S\times\{0\}$ is homotopic to $h$.
\end{lemma} 

\begin{proof} 
If the Reeb vector field is tangent to $S\times\{t\}$ in $(p,t)$ then
\[\iota_{R_\alpha}((1-F(t))\d\beta+F(t)h^*\d\beta)(p,t)=0\] 
as \[\d\alpha=(1-F(t))\d\beta+F(t)h^*\d\beta+F'(t)\d t\wedge(h^*\beta-\beta).\]
Yet $\d\beta$ and $h^*\d\beta$ are two positive volume forms. Hence $R_\alpha$ is transverse to $S\times\{t\}$. It is positively transverse by the boundary condition. The first return map is well defined and homotopic to $h$ as $h$ is the first return map of $\frac{\partial}{\partial t}$ on $S\times\{0\}$ and $R_\alpha$ and $\frac{\partial}{\partial t}$ are homotopic in the space of vector fields transverse to $S\times\{*\}$. 
\end{proof} 

In $M_0$, periodic Reeb orbits correspond to periodic points of the first return map on $S\times\{0\}$. Without loss of generality, we may assume that all the periodic points of the first return map in the interior of $S$ are non-degenerate.  

\subsubsection{In non-hyperbolic components}
We use the following theorem of Colin and Honda.

\begin{theorem}[Colin-Honda, {\cite[Théorème 1.3]{ColinHonda05}}] \label{theoreme_non_hyperbolique}
Let $M$ be a compact, oriented, irreducible $3$-manifold with non-empty boundary such that $\partial M$ is a union of tori. Then there exist an hypertight contact form $\alpha$ on $M$ and a neighborhood  $T\times I$ of each boundary component ($I$ is an interval) with coordinates $(x,y,z)$ such that $\alpha=\cos(z)\d x-\sin(z)\d y$. In addition there exist arbitrarily small non-degenerate hypertight perturbations of~$\alpha$. 
\end{theorem} 
The construction in \cite{ColinHonda05} gives the same contact structures as \cite{HKM00} and \cite{Colin01}. 
Without loss of generality, all the periodic orbits whose free homotopy classes do not correspond to a class in the boundary are non-degenerate.

\subsubsection{Interpolation and torsion} 
In the previous sections we constructed an hypertight contact form $\alpha$ on each connected component of $M\setminus \bigcup_{k=1}^N \nu(T_k)$ where $\nu(T_k)$ is a neighborhood of $T_k$. We now glue these components together. Choose $k\in1\dots N$. There exist coordinates $(x,y,z)$ in a neighborhood $T_k\times [a,b]$ of $T_k$ such that in a neighborhood of $T\times\{a\}$ the contact form is written $f_a(x)\d y+g_a(x)\d z $ and in a neighborhood of $T\times\{b\}$ the contact form is written $f_b(x)\d y+g_b(x)\d z $.
\begin{lemma}
For all $n\in \bb N^*$,
there exist $f_n:[a,b]\longrightarrow\bb R$ and $g_n:[a,b]\longrightarrow\bb R$ two smooth functions such that
\begin{enumerate}
  \item $f_n$ extends $f_a$ and $f_b$ and $g_n$ extends $g_a$ and $g_b$;
  \item\label{cond_alpha} $\alpha=f_n(x)\d y+g_n(x)\d z$ is a contact form;
  \item in coordinates\label{cond_R_alpha} $(\theta,z)$, the Reeb vector field $R_\alpha$ sweeps out an angle in \[\left(2n\pi-\frac{\pi}{2},2n\pi+\frac{3\pi}{2}\right].\]
\end{enumerate}
\end{lemma}
\begin{proof}
The contact condition is $f_n'g_n-f_ng_n'>0$ and the Reeb vector field is
\[R_\alpha=\frac{1}{f_n'g_n-f_ng_n'}\left(
\begin{array}{l}
0\\
-g_n'\\
f_n'
\end{array}\right).\]
The conditions \eqref{cond_alpha} and \eqref{cond_R_alpha} are equivalent to ``the parametrized curve $(f_n,g_n)$ in $\mathbb R^2$ turns clockwise and its normal vector sweeps out an angle in $(2n\pi-\frac{\pi}{2},2n\pi+\frac{3\pi}{2}]$''. We choose a parametric curve in $\mathbb R^2$ extending $(f_a,g_a)$ and $(f_b,g_b)$ with these properties.
\end{proof}
For all $n\in \bb N^*$, construct a contact form $\alpha_n$ on $M$ by extending $\alpha$ by $\alpha_n=f_n(x)\d y+g_n(x)\d z$ in a neighborhood of $T_1$ and by $\alpha_n=f_1(x)\d y+g_1(x)\d z$ in a neighborhood of $T_2,\dots, T_N$.

\begin{remark}
If $\{b\}\times T$ is in $\partial M_0$, then $f_b<0$, $f_b'>0$ and $g_b=1$ near $b$. If $\{a\}\times T$ is in $\partial M_0$, then $f_a>0$, $f_a'>0$ and $g_a=1$ near $a$ (changes in signs are due to the orientation convention of the boundary).
\end{remark}
By \cite[Théorème 4.2]{Colin99}, as contact structures $\xi_n=\ker(\alpha_n)$ are universally tight on each components, $(M,\xi_n)$ is universally tight for all $n\in\bb N^*$. In addition, as our construction corresponds to the construction in \cite[Section 4]{Colin01}, by Theorem \cite[ Theorem 4.5]{Colin01}, there exist infinitely many non-isomorphic $\xi_n$.

\subsection{Growth rate of contact homology}\label{section_HC_pA}

\begin{lemma}\label{cylindres_hyperboliques}
For all almost complex structures on $\mathbb R\times M $ adapted to the contact form constructed above, there is no holomorphic cylinder $u:\mathbb R\times S^1\to\mathbb R\times M$ asymptotic to two periodic Reeb orbits contained in different connected components of $M\setminus (\bigcup_{k=1}^N T_k\times[a,b])$.
\end{lemma}

\begin{proof}
In $T_k\times[a,b]$, the contact form is written $\alpha_n=f(x)\d y+g(x)\d z$ and the Reeb vector field is
\[R_{\alpha_n}=\frac{1}{f'g-fg'}\left(
\begin{array}{l}
0\\
-g'\\
f'
\end{array}\right).\]
It depends only on the $x$ variable and we denote it by $R_{\alpha_n}(x)$.
We prove this result by contradiction. If such a $u$ exists then there exists $k\in 1\dots N$ such that $u_M(\mathbb R\times S^1)\cap T_k\times\{x\}\neq\emptyset$ for all $x\in [a,b]$.
By Theorem \ref{theoreme_zeros_isoles} and Proposition \ref{proposition_tau_im_du}, there exist $x_0$ and $x_1$ in $[a,b]$ such that
\begin{enumerate}
  \item $R_{\alpha_n}(x_0)=-R_{\alpha_n}(x_1)$,
  \item for all $(s,t)\in C=u^{-1}\left(u(\mathbb R\times S^1)\cap(\mathbb R\times \{x_0,x_1\}\times T)\right)$, we have $\d u(s,t)\neq 0$ and $\frac{\partial}{\partial \tau}\notin \im(\d u(s,t))$.
\end{enumerate}
By Lemma \ref{lemme_cylindres_feuilletes}, $C$ is a finite union of smooth circles homotopic to $\{*\}\times S^1$. Cut $\mathbb R\times S^1 $ along these circles and choose a connected component $\Sigma$ such that $u_M(\Sigma)\cap\{x\}\times T\neq \emptyset$ for all $x\in[x_0,x_1]$. Then $\partial u_M(\Sigma)$ is a union of two homotopic circles: one in $\{x_0\}\times T$ and one in $\{x_1\}\times T$. By positivity of intersection the Reeb vector field is positively transverse to these circles in $\{x_0\}\times T$ and $\{x_1\}\times T$. This leads to a contradiction as $R_{\alpha_n}(x_0)=-R_{\alpha_n}(x_1)$.
\end{proof}

Let $\Lambda_0$ be the set of primitive free homotopy classes that correspond to periodic orbits in $M_0$ and do not represent a homotopy class in a torus $T_k$ for $k=1,\dots, N$. All the periodic Reeb orbits with homotopy class in $\Lambda_0$ are non-degenerate. As there are no contractible periodic orbits, the associated partial contact homology is well defined.

There exists $C>0$ such that all periodic orbits in $M_0$ associated to a $k$-periodic point of the first return map $h_1$ have a period smaller than $kC$.

\begin{lemma}\label{lemma_lambda_0}
For all $\eta\in \Lambda_0$,  $\dim(HC^{[\eta]}_*(M,\alpha_n))\geq 1$. In addition, if $\eta$ is associated to $k$-periodic points, for all $L>kC$ the map
$HC^{[\eta]}_{\leq L}(M,\alpha_n)\to HC^{[\eta]}_*(M,\alpha_n)$ has a rank greater than $1$.  
\end{lemma}  

\begin{proof} 
Choose $\eta\in\Lambda_0$.
Write $C^\eta_*=C_0\oplus C_1$ where $C_0$ is generated by periodic orbits in $M_0$ homotopic to $\eta$ and $C_1$ is generated by periodic orbits in $M\setminus M_0$ homotopic to $\eta$. By Lemma \ref{cylindres_hyperboliques}, the differential is written
\[\left(\begin{array}{cc} 
\partial_\eta & 0 \\ 
0& * 
\end{array}\right).\] 
We prove that $\dim\left(\ker(\partial_\eta)/\im(\partial_\eta)\right)\geq 1$. Write $C_0=E\oplus O$ where $E$ is generated by even periodic orbits and $O$ by odd periodic orbits (as $\eta$ is primitive all the periodic orbits are good). Then
\[\partial_\eta=\left(\begin{array}{cc}
0 & \partial_O \\
\partial_E& 0
\end{array}\right)\]
and \[\ker(\partial_\eta)/\im(\partial_\eta)=\ker(\partial_E)/\im(\partial_O)\oplus\ker(\partial_O)/\im(\partial_E).\]
Hence, $ \dim\left(\ker(\partial_\eta)/\im(\partial_\eta)\right)=\{0\}$ if and only if $\dim(\ker(\partial_E))=\dim(\im(\partial_O))$ and $\dim(\ker(\partial_O))=\dim(\im(\partial_E))$.

By Section \ref{proprietes_pA}, there exist a branched cover $\hat S$ of $S$ and a pseudo-Anosov map $\hat \psi$ such that the lift $\hat h_1$ of $h_1$ is homotopic to $\hat \psi$. Let $c$ denote the Nielsen class associated to the periodic orbits in $M_0$ homotopic to $\eta$ and $k$ denote the order of the associated periodic points. Let $\hat c$ be a Nielsen class of $\hat h_1$ containing a lift of a point in $c$. As $c$ does not contain points in $\partial S$, all periodic points in $\hat c$ are non-degenerate and there exists $s$ such that $\hat c$ contains exactly $s$ lifts of each point in $c$. By Theorems \ref{theoreme_Nielsen_pA} and \ref{theoreme_invariance_Nielsen}, $\Lambda_{\hat h_1^k}(\hat c)=\Lambda_{\hat \psi^k}(\hat c)\neq 0$. A periodic point of $h_1$ is even if and only if the associated Reeb orbit is even. Therefore, we have $\Lambda_{\hat h_1^k}(\hat c)=s\dim(E)-s\dim(O)$ and \[\dim(\ker(\partial_O))+\dim(\im(\partial_O))\neq\dim(\ker(\partial_E))+\dim(\im(\partial_E)).\] Hence, 
$\dim\left(\ker(\partial_\eta)
/\im(\partial_\eta)\right)>0$ and $\dim(HC^{[\eta]}_*(M,\alpha_n))\geq 1$.

For all $L>kC$, write $C^\eta_{\leq L}=C_0\oplus C_{\leq L}$. As the differential is written
\[\left(\begin{array}{cc}
\partial_\eta & 0 \\
0& *
\end{array}\right)\]
the map
$HC^{[\eta]}_{\leq L}(M,\alpha_n)\to HC^{[\eta]}_*(M,\alpha_n)$ has a rank greater than $1$.
\end{proof}

\begin{proof}[Proof of Theorem \ref{theoreme_hyperbolique_homologie}]
It remains to prove that the growth rate of $HC^{\Lambda_0}_*(M,\alpha_n)$ is exponential. By Theorems \ref{theoreme_croissance_pA}, \ref{theoreme_Nielsen_pA}, \ref{theoreme_invariance_Nielsen} and \ref{theoreme_pA_boundary}, the number of Nielsen classes associated to periodic points of the first return map $h_1$ grows exponentially. As the number of homotopy classes in tori $T_k$ for $k=1\dots N$ exhibits a quadratic growth (Lemma~\ref{lemme_croissance_sigma}) and by Lemma \ref{lemma_lambda_0}, the growth rate of partial cylindrical homology is exponential. 
\end{proof}

\begin{proof}[Proof of Theorem \ref{theoreme_hyperbolique}]
By invariance of the growth rate of partial contact homology (Proposition \ref{proposition_invariance_croissance_partielle}), the growth rate of the number of periodic Reeb orbits is exponential if cylindrical contact homology is well defined (Remark \ref{remark_growth}), i.e. if the contact form is non-degenerate and hypertight.

Under Hypothesis H, let $\alpha^p_n$ be a non-degenerate contact form such that $\xi_n=\ker(\alpha^p_n)$. By Theorem \ref{theoreme_non_hyperbolique} and Lemma \ref{noncontractile} there exists an hypertight non-degenerate contact form $\alpha'_n$ of $\xi_n$. By Theorem \ref{theoreme_hyperbolique_homologie} the growth rate of cylindrical contact homology is exponential. Consider the map $A_*(M,\alpha^p_n, J)\to A_*(M,\alpha'_n, J')$ given by Theorem \ref{theoreme_Phi_0_1} and the pull back augmentation induced by the trivial augmentation on $A_*(M,\alpha'_n, J')$ (Proposition \ref{proposition_trivial_augmentation}). By invariance of the growth rate of linearized contact homology (Proposition \ref{proposition_invariance_croissance_linearisee}), $N_L(\alpha^p_n)$ exhibits an exponential growth.
\end{proof}

\end{document}